\documentclass[11pt,twoside]{amsart}
\usepackage{epic}
\usepackage[latin1]{inputenc}
\usepackage[OT1]{fontenc}
\usepackage{amsmath}
\usepackage{amsthm}
\usepackage{amssymb}
\usepackage[all]{xy}
\xyoption{curve}
\usepackage{ifthen} 
\usepackage{hyperref} 
\usepackage{graphicx}

\newtheorem{thm}{Theorem}[section]

\newtheorem{prop}[thm]{Proposition}

\newtheorem{lem}[thm]{Lemma}
\newtheorem{cor}[thm]{Corollary}

\numberwithin{equation}{section}

\theoremstyle{definition}
\newtheorem{definition}[thm]{Definition}
\newtheorem{remark}[thm]{Remark}

\newtheorem{notation}[thm]{Notation}

\newcommand{\qqed}{\hspace*{\fill}$\Box$}

\newcommand{\rank}{\operatorname{rank}}

\newcommand{\Db}{{\rm D}^{\rm b}}

\newcommand{\CH}{{\rm CH}}

\newcommand{\Pic}{{\rm Pic}}

\newcommand{\rk}{{\rm rk}}

\newcommand{\ch}{{\rm ch}}

\newcommand{\ko}{{\mathcal O}}

\newcommand{\IC}{\mathbb{C}}

\newcommand{\IN}{\mathbb{N}}
\newcommand{\IP}{\mathbb{P}}
\newcommand{\IQ}{\mathbb{Q}}

\newcommand{\IZ}{\mathbb{Z}}
\newcommand{\ZZ}{\mathbb{Z}}
\newcommand{\PP}{\mathbb{P}}

\DeclareMathOperator{\Gal}{Gal}

\renewcommand{\to}{\xymatrix@1@=15pt{\ar[r]&}}
\renewcommand{\rightarrow}{\xymatrix@1@=15pt{\ar[r]&}}
\renewcommand{\mapsto}{\xymatrix@1@=15pt{\ar@{|->}[r]&}}
\renewcommand{\twoheadrightarrow}{\xymatrix@1@=15pt{\ar@{->>}[r]&}}
\renewcommand{\hookrightarrow}{\xymatrix@1@=15pt{\ar@{^(->}[r]&}}
\newcommand{\congpf}{\xymatrix@1@=15pt{\ar[r]^-\sim&}}
\renewcommand{\cong}{\simeq}

\newcommand{\Fermat}{Y}
\newcommand{\Group}{\IZ/5\IZ}

\newcommand{\Godeaux}{X}
\newcommand{\betacom}{B}
\newcommand{\autGodeaux}{\mathrm{Aut}(\Godeaux)}
\newcommand{\golden}{\Phi}

\begin{document}

\newboolean{xlabels} 
\newcommand{\xlabel}[1]{ 
                        \label{#1} 
                        \ifthenelse{\boolean{xlabels}} 
                                   {\marginpar[\hfill{\tiny #1}]{{\tiny #1}}} 
                                   {} 
                       } 
\setboolean{xlabels}{false} 

\title[Derived category of Godeaux surface]{On the derived category of the classical Godeaux surface}

\author[B\"ohning]{Christian B\"ohning$^1$}
\address{Christian B\"ohning, Fachbereich Mathematik der Universit\"at Hamburg\\
Bundesstra\ss e 55\\
20146 Hamburg, Germany}
\email{christian.boehning@math.uni-hamburg.de}

\author[Bothmer]{Hans-Christian Graf von Bothmer}
\address{Hans-Christian Graf von Bothmer, Fachbereich Mathematik der Universit\"at Hamburg\\
Bundesstra\ss e 55\\
20146 Hamburg, Germany}
\email{hans.christian.v.bothmer@uni-hamburg.de}

\author[Sosna]{Pawel Sosna$^2$}
\address{Pawel Sosna, Fachbereich Mathematik der Universit\"at Hamburg\\
Bundesstra\ss e 55\\
20146 Hamburg, Germany}
\email{pawel.sosna@math.uni-hamburg.de}

\thanks{$^1$ Supported by Heisenberg-Stipendium BO 3699/1-1 of the DFG (German Research Foundation)}
\thanks{$^2$ Supported by the RTG 1670 of the  DFG (German Research Foundation)}

\begin{abstract}
We construct an exceptional sequence of length $11$ on the classical Godeaux surface $X$ which is the $\Group$-quotient of the Fermat quintic surface in $\mathbb{P}^3$. This is the maximal possible length of such a sequence on this surface which has Grothendieck group $\mathbb{Z}^{11} \oplus \Group$. In particular, the result answers Kuznetsov's Nonvanishing Conjecture, which concerns Hochschild homology of an admissible subcategory, in the negative. The sequence carries a symmetry when interpreted in terms of the root lattice of the simple Lie algebra of type $E_8$.  We also produce explicit nonzero objects in the (right) orthogonal to the exceptional sequence.
\end{abstract}

\maketitle

\section{Introduction}\xlabel{sIntroduction}

The bounded derived category of coherent sheaves $\Db(X)$ on a smooth projective variety $X$ (always over $\mathbb{C}$ in the following) may be viewed as a categorification of the Grothendieck group of $X$ or the Chow ring of $X$, both of which tend to be very intricate objects in their own right already. Moreover, see, for example, \cite{Kontsevich}, \cite{Orlov}, \cite{Tabuada}, there is the intuition that $\Db(X)$ should be a version of the non-commutative motive of $X$, with decompositions of the (classical) Chow motive $h(X)$ of $X$ being reflected in a suitable sense by semi-orthogonal decompositions of $\Db (X)$. Recently (see, for example, \cite{Rouquier}, \cite{Kawamata}, \cite{Kuz10}) a number of results as well as conjectures try to link semi-orthogonal decompositions in derived categories to the birational geometry of $X$, including very subtle features such as the rationality or irrationality of $X$ which do not seem to be detected by sheaf-cohomological (non-categorical) data. However, the best understood examples considered so far mainly consist of varieties close to the toric and rational-homogeneous ones as well as some Fano hypersurfaces. We feel that in many ways the optimism radiated by existing conjectures is not reflected in the data one can sample from these varieties. In this paper we study the Godeaux surface. 
\smallskip

It follows from Serre duality that the derived category of a Calabi-Yau manifold is indecomposable, that is, it does not admit any non-trivial semi-orthogonal decomposition. Furthermore, varieties of general type with globally generated canonical bundle do not have exceptional objects, see \cite{Okawa}. However, on surfaces of general type $X$ with $p_g=q=0$ every line bundle is exceptional and one may hope that interesting semi-orthogonal decompositions exist which may yield a nontrivial testing ground for existing conjectures. See \cite{BCP} for a survey on these surfaces in general as well as further references. In this paper we study the classical Godeaux surface $X$ which is the $\Group$-quotient of the Fermat hypersurface in $\mathbb{P}^3$ given by $x_1^5 + x_2^5 + x_3^5 + x_4^5 =0$ by the action $(x_1:x_2:x_3:x_4) \mapsto (\xi^1 x_1: \xi^2 x_2 : \xi^3 x_3 : \xi^4 x_4)$ for $\xi \in \Group$. One knows that the Chow motive of this surface splits as a direct sum of Lefschetz motives \cite{GulPed}
\[
h(X) \simeq 1 \oplus 9\mathbb{L} \oplus \mathbb{L}^2
\]
and hence that $X$ has the same motive as a rational surface with the same Betti numbers. The Grothendieck group of $X$ is $\mathbb{Z}^{11}\oplus \Group$, hence $X$ does not admit a full exceptional sequence, since the existence of the latter would imply that the Grothendieck group is free. One may therefore conjecture that $\Db (X)$ has an exceptional sequence of length $11$ corresponding to the ``trivial commutative part of the motive" and some nontrivial genuinely non-commutative semi-orthogonal complement to this sequence. This expectation turns out to be correct and is the main result of this paper:\smallskip

\noindent
\textbf{Theorem \ref{tMain}}
\textit{Let $X$ be the classical Godeaux surface. There exists a semi-orthogonal decomposition
\[
\Db (X) = \left\langle \mathcal{A},\mathcal{L}_1, \dots ,\mathcal{L}_{11} \right\rangle 
\]
where $( \mathcal{L}_1, \dots ,\mathcal{L}_{11} )$ is an exceptional sequence of maximal length consisting of line bundles on $X$ and $\mathcal{A}\neq 0$ is the right orthogonal to this sequence.}
\smallskip

Note, however, that on the categorical level, $\Db(X)$ is certainly not equivalent to the derived category of a rational surface, because derived equivalent varieties have equal Kodaira dimension, see \cite{Orlov3}.\smallskip

Here is the roadmap of the paper:  in Sections \ref{sKgroup} and \ref{sAuto} we assemble some background material on the Grothendieck group, cohomology and automorphisms of the Godeaux surface $X$, all of which is more or less well known. Section \ref{sDegree1} can be said to be the technical heart of the paper: we completely classify effective degree $1$ divisors on $X$. Here degree always means degree with respect to $K_X$. In particular, we obtain a configuration of fifty  $\Group$-invariant elliptic quintic curves on the Fermat surface $Y$ (this configuration is obtained by the  methods of \cite{Reid}), and describe the intersection pairing between their images on the Godeaux $X$. In Section \ref{sE8} we describe the $E_8$-symmetry of the situation explicitly: the group of divisors modulo numerical equivalence $N(X)$ on $X$ is a lattice of type $\mathbf{1}\oplus (-E_8)$ with $-E_8$ the negative of the root lattice of the simple Lie algebra of type $E_8$.  We use the $E_8$-symmetry to produce a numerically semiorthogonal sequence of vectors in the Grothendieck group of $X$. 
In Section \ref{sTower} we review the Campedelli model of the Godeaux surface, illustrate how it provides a convenient set-up for calculations with line bundles on $X$ in terms of data on $\mathbb{P}^1\times \mathbb{P}^1$, and prove some vanishing results. The short Section \ref{sTorsion} explains the role of the torsion in the Picard group of $X$.  In Section \ref{sExceptional} we lift the above numerically semiorthogonal sequence to the exceptional sequence $\mathcal{L}_1, \dots ,\mathcal{L}_{11}$ using the vanishing results of Sections \ref{sTower} and \ref{sTorsion}. 

The existence of the above decomposition answers Kuznetsov's Nonvanishing Conjecture about the Hochschild homology of an admissible subcategory, \cite[Conj.\ 9.1]{Kuz09}, in the negative. In fact, the Hochschild homology of $\mathcal{A}$ is zero, but $\mathcal{A}$ itself is not.

In Section \ref{sOrthogonal}, we produce some explicit nonzero objects in the complement $\mathcal{A}$. They
arise as mapping cones of morphisms $\mathcal{O}_X \to \mathcal{O}_{\tau}[2]$ where $\mathcal{O}_{\tau}$ is a nontrivial torsion line bundle on $X$. This is based on the fact that the triangulated subcategory generated by $\mathcal{O}_{\tau}$ and $\mathcal{O}$ is in the right-orthogonal to $\left\langle \mathcal{L}_2, \dots ,\mathcal{L}_{11}\right\rangle $. 
It would be interesting, but probably require some additional ideas, to completely describe $\mathcal{A}$ as a category. 

In Section \ref{sRigidity} we investigate the behaviour of the subcategory $\langle \mathcal{L}_1, \dots , \mathcal{L}_{11}\rangle$ when $X=X_0$ varies in the family of $\ZZ/5$-torsion numerical Godeaux surfaces (all of whose canonical models are quotients of $\ZZ/5$-invariant quintics in $\PP^3$): if $X_t$, $t$ a deformation parameter, is a surface sufficiently close to $X_0$ in the family, then we also have a decomposition $\mathrm{D}^b (X_t ) = \langle \mathcal{A}_t, \mathcal{L}_{1,t}, \dots , \mathcal{L}_{11, t} \rangle$ where $\mathcal{L}_{i,t}$ are line bundles which are deformations of the $\mathcal{L}_i$, and $\mathrm{K}_0 (\mathcal{A}_t)= \ZZ/5$. Moreover, the subcategory $\langle \mathcal{L}_{1, t}, \dots , \mathcal{L}_{11, t}\rangle$ does not vary in a small neighbourhood of the generic point of the family: all these categories are equivalent there.  This is the same phenomenon as observed in the paper \cite{A-O12} which appeared shortly after the first version of the present paper was posted, and was an important inspiration for us to add Section \ref{sRigidity} to the text.
 
This paper makes frequent use of computer algebra computations. Most are linear algebra or combinatorial and could be done by hand, but we do them on a computer for convenience. Genuine Gr\"obner basis computations are needed in the classification of elliptic curves on $X$ and in Lemma \ref{l:TorsionEL}.
 
\noindent
\textbf{Acknowledgements.} We thank Marcello Bernardara, Fabrizio Catanese, Claudio Pedrini and Helge Ruddat for useful discussions and Sven Porst for procuring the essential \cite{Reid}. We thank Sergey Galkin, Ludmil Katzarkov, Alexander Kuznetsov and Dmitri Orlov for comments and discussions on the first version of the text. 

\section{The cohomology and K-theory of the Godeaux surface}\xlabel{sKgroup}

Let $X = Y/(\Group)$ be the Godeaux surface where $Y:= \{ x_1^5 + \dots + x_4^5 =0 \} \subset \mathbb{P}^3$ is the Fermat (hyper)surface with action of a generator $\xi \in \Group$ given by
\[
(x_1:x_2:x_3:x_4) \mapsto (\xi^1 x_1: \xi^2 x_2 : \xi^3 x_3 : \xi^4 x_4)\, .
\]
We denote the quotient map by $p\colon Y \to X$. We have $K_X^2 = 1$, $p_g=q=0$, hence $\mathrm{Pic} (X) \simeq H^2 (X,\mathbb{Z})$ by the exponential sequence, and the integral homology and cohomology of $X$ are as follows:

\begin{gather*}
H_i (X,\mathbb{Z})= H^i (X,\mathbb{Z}) = \mathbb{Z}\; \mathrm{for} \; i=0,4, \\ \quad H_1 (X,\mathbb{Z} ) = \pi_1 (X)^{\mathrm{ab}} = \Group \simeq H^3 (X,\mathbb{Z} )
\end{gather*}
and
\begin{gather*}
H_2 (X,\mathbb{Z} ) \simeq H^2 (X,\mathbb{Z} ) \simeq (H_2(X,\IZ)/\text{tors})\oplus \text{tors}(H_1(X,\IZ)) \simeq \mathbb{Z}^9\oplus \Group\, , \\
H^1 (X,\mathbb{Z}) \simeq H_3 (X,\mathbb{Z}) \simeq \{ 0 \}\, .
\end{gather*}

Thus not all integral cohomology is algebraic on $X$ (because $H^3 (X,\mathbb{Z})$ is nontrivial), but rationally this is true. The intersection pairing
\[
\langle \cdot , \cdot \rangle\colon N(X) \times N(X) \to \mathbb{Z}
\]
on the group of divisors modulo numerical equivalence $N(X)\simeq \mathrm{Pic} (X) /(\mathrm{tors})$ is a nondegenerate unimodular pairing.
\smallskip

Recall the following facts about the Grothendieck group $K(W)$ of a smooth variety $W$. 
The Chern character defines a ring homomorphism
\[\ch \colon K(W) \rightarrow \text{CH}^*(W)_\IQ,\]
where $\text{CH}^*(W)_\IQ=\text{CH}^*(W)\otimes\IQ$ and $\text{CH}^*(W)$ is the Chow ring of cycles graded by codimension. In fact, the Chern character defines an isomorphism between $K(W)_\IQ$ and $\text{CH}^*(W)_\IQ$. If $W$ is a surface, then there are the following isomorphisms (the first two are true in any dimension)
\[\rk\colon F^0K(W)/F^1K(W)\cong \CH^0(W)\cong\IZ,\]
\[c_1\colon F^1K(W)/F^2K(W)\cong \Pic(W),\]
\[c_2\colon F^2K(W)\cong \text{CH}^2(W),\]
where $F^iK(W)$ is the subgroup generated by sheaves with support of codimension $\geq i$ (see \cite[Ex.\ 15.3.6]{Fulton}). Thus, if $\ch\colon K(W)\rightarrow \text{CH}^*(W)$ is defined over $\IZ$, then it is automatically injective. Note that $\text{CH}^0(W)\oplus\text{CH}^2(W)$ is always contained in the image of $\ch$. If $\frac{1}{2}c_1(\mathcal{L})^2 \in \IZ$ for any line bundle $\mathcal{L}\in \Pic(W)$, then $\ch$ is easily seen to be surjective as well. 

For the Godeaux surface, however, the Chern character is not integral. Indeed, $\ch(K_X)=(1,K_X,\frac{1}{2})$. Still, the Bloch conjecture holds for the Godeaux surface and hence $\CH^2(X)\cong \IZ$. We have the following

\begin{prop}\label{grothgroup}
The Grothendieck group of the Godeaux surface is isomorphic to $\IZ^{11}\oplus \IZ/5\IZ$.
\end{prop}

\begin{proof}
Consider the above isomorphisms. For the Godeaux surface we have $H^2(X,\IZ)\cong\Pic(X)\cong \IZ^9\oplus \IZ/5\IZ$, because $H^2(X,\IZ)\cong (H_2(X,\IZ)/\text{tors})\oplus \text{tors}(H_1(X,\IZ))$, see \cite[Thm.\ IV.3.5]{FrGr}. Hence, we have an exact sequence
\[\begin{xy}\xymatrix{ 0 \ar[r] & \IZ \ar[r] & F^1K(X) \ar[r] & \IZ^9\oplus \IZ/5\IZ \ar[r] & 0.}\end{xy}\]
Clearly, $\rk(F^1K(X))\geq 10$. If we had strict inequality, then $\rk(F^0K(X))=\rk(K(X))\geq 12$, giving a contradiction, since we know that the latter rank is 11. Therefore $F^1K(X)\cong \IZ^{10}\oplus T$ for some finite abelian group $T$ because the class of the structure sheaf $\mathcal{O}_p$ of a point $p\in X$, which is a generator of $\mathrm{CH}^2 (X)$, is primitive in $\mathrm{K}_0 (X)$, for example because $\chi (\mathcal{O}_p, \mathcal{O}_X ) = 1$. 
A similar argument with the exact sequence involving $F^1$ and $F^0$ finishes the proof.
\end{proof}

\section{The automorphism group of the Godeaux surface}\xlabel{sAuto}
The automorphism group $\mathrm{Aut}(X)$ of the classical Godeaux surface $X$ has been determined in \cite{Maggiolo} and 
is generated by the diagonal actions
\begin{align*}
		   \gamma(x_1,x_2,x_3,x_4) &:= (\xi x_1,x_2,x_3,\xi^{-1} x_4)\\
		   \delta(x_1,x_2,x_3,x_4) &:= (\xi x_1,\xi^{-1}x_2,\xi^{-1}x_3,\xi x_4)
\end{align*}
and the permutation
\[
		  \alpha (x_1,x_2,x_3,x_4) := (x_3,x_1,x_4,x_2).
\]
Notice that $\beta := \alpha^2$ is an involution. As an abstract group we have $\mathrm{Aut}(X) \simeq \mathbb{Z}/4 \ltimes (\Group)^2$. Indeed, it is a quotient of the group $\IZ/4\ltimes (\IZ/5)^4$, where $\alpha$ permutes the copies in $(\IZ/5)^4$, by the normal subgroup generated by $(\xi, \xi, \xi, \xi)$ and $(\xi, \xi^2, \xi^3 , \xi^4 )$. 

\section{The effective degree $1$ divisors on $X$}\xlabel{sDegree1}

In this section we classify degree $1$ curves on the Godeaux surface. This is equivalent to classifying reduced curves $C$ of degree $5$ on the Fermat quintic $\Fermat \subset \IP^3$.
 
If $C$ is reducible, then equivariance implies that $C$ must be the  union of $5$ lines that form an orbit under the operation of $\Group$. The lines on the Fermat surface are classified
by

\begin{thm}\xlabel{tLines}
There are exactly $75$ lines on $Y$. Let $\xi$ be a primitive fifth root of unity. 
The following are parameter forms for the lines:
\begin{align*}
l^- (i,j ) &= \{ [s:\: - \xi^i s :\: t :\: -\xi^j t] \, \mid\, [s:\: t]\in \mathbb{P}^1 \}\, ,\\
l^0 (i,j ) &= \{ [s:\: t :\: - \xi^j t :\: -\xi^i s ] \, \mid\, [s:\: t]\in \mathbb{P}^1 \}\, ,\\
l^+ (i,j ) &= \{ [s:\: t :\: - \xi^i s :\: -\xi^j t ] \, \mid\, [s:\: t]\in \mathbb{P}^1 \}.
\end{align*}
Here, $i,j\in\{ 0, \dots , 4\}$.
\end{thm}

\begin{proof}
By \cite[Prop.\ 3.3]{BS}, there are exactly $75$ such lines. The explicit form is taken from \cite[Sect.\ 3]{SSvL}.
\end{proof}
\smallskip

\begin{notation}
Representatives of the $\Group$-orbits of lines are given by $l^{+/0/-}(0,j)$ with $j\in \ZZ/5$. We denote the images of these lines on the Godeaux surface by $L^{+/0/-}_j$ and also call them ``lines". 
\end{notation}

\begin{prop}\xlabel{p:intersectionLine}
The intersection pairing on the $15$ lines on the Godeaux surface is
given by

\begin{center}
\begin{tabular}{c|ccc}
	$L^m_i. L^n_j$ & &  $m-n$ \\ 
		                        & 0 & $\pm1$ & $\pm2$ \\ \hline
	$i =j$		     & -3 & 1 & 5 \\
	$i\not=j$		     & 2 & 1 & 0 \\ 
\end{tabular} 
\end{center}
where we interpret the upper indices $+/0/-$ as $+1/0/-1$.
\end{prop}

\begin{proof}
This can be done by direct calculation. A Macaulay2 script doing this can be found at \cite{BBS}. The symmetry is explained by the operation of $\autGodeaux$ and $\Gal(\IQ(\xi):\IQ)$.
\end{proof}

We now assume that $C \subset \Fermat \subset \IP^3$ is reduced and irreducible. 
If $C$ is degenerate, i.e. contained in a hyperplane, we obtain the $4$ invariant plane quintics cut out by $x_1, x_2, x_3$ and $x_4$.

If $C$ is reduced, irreducible and nondegenerate one can use

\begin{thm}[Gruson-Lazarsfeld-Peskine, \cite{GLP}]
Let $C$ be a nondegenerate reduced irreducible space curve of degree $d$. Then $C$ is $(d - 1)$-regular. Moreover if $d \ge 5$, $C$ is not $(d - 2)$-regular if and only if $C$ is a smooth rational curve with a $(d - 1)$-secant line. 
\end{thm}

If we exclude the second case for the moment, so assume that $C$ is $3$-regular,  we have

\begin{prop}
Let $C \subset \IP^3$ be an irreducible reduced space curve of degree $5$ and regularity $3$ that does not
lie on a quadric. Then $C$ is an elliptic curve.
\end{prop}

\begin{proof}
This proposition is a well known fact, but we will sketch the proof for the convenience of the reader. We will use a combinatorial tool developed by Green \cite{Green}, especially section 4. By considering generic initial ideals, Green associates to each irreducible space curve a function
\[
	f_C : \IN \times \IN \to \IN \cup \{\infty\}
\]
that he depicts graphically by writing $f_C(i,j)$ at position $(i,j)$ in a triangular diagram:
\[
	\begin{matrix}
	   &&   && f(0,0) &&   && \\
	   &   && f(0,1) && f(1,0) &&   & \\
	   && f(0,2) && f(1,1) && f(2,0) && \\
	   & f(0,3) && f(1,2) && f(2,1) && f(3,0)& \\
	   &&   && \dots &&   && \\
	\end{matrix}
\]
For example:
\[
	\begin{matrix}
	   &&   && \infty &&   && \\
	   &   && \infty && \infty &&   & \\
	   && 1 && \infty && \infty && \\
	   & 0 && 0 && 0 &&  0& \\
	   &&   && \dots &&   && \\
	\end{matrix}
\]
Sometimes one  replaces $\infty$ by circles and $0$ by crosses.

If such a diagram comes from an irreducible reduced space curve it satisfies
a number of conditions:
\begin{enumerate}
\item The numbers are weakly decreasing from top to bottom and from right to left.
\item The number $d$ of $\infty$'s is equal to the degree of the curve, i.e. $d=5$ in the example above.
\item Replace an $\infty$ at $(i,j)$ by $i+j-1$ and replace a number $f(i,j) \not=\infty$ by $-f(i,j)$. Then the sum of the entries of the diagram plus 1 is the arithmetic genus $g_a$ of the curve. The example above gives
\[
	\begin{matrix}
	   &&   && -1 &&   && \\
	   &   && 0 && 0 &&   & \\
	   && -1 && 1 && 1 && \\
	   & 0 && 0 && 0 &&  0& \\
	   &&   && \dots &&   && \\
	\end{matrix} \quad
\]
so the arithmetic genus of a curve with this diagram is $g_a = 1$.
\item Let $r$ be the number of rows that contain non zero entries and replace the numbers $\not = 0,\infty$ at $(i,j)$ by $f(i,j)+i+j$ and $\infty$ by $0$.  Then the maximum of all entries and $r$ is the Castelnuovo-Mumford regularity of the curve. The example above gives
\[
	\begin{matrix}
	   &&   && 0 &&   && \\
	   &   && 0 && 0 &&   & \\
	   && 3 && 0 && 0 && \\
	   & 0 && 0 && 0 &&  0& \\
	   &&   && \dots &&   && \\
	\end{matrix} \quad
\]
and has therefore regularity $3$.
\item Let $r'$ be the smallest number such that $r' = i+j$ and $f(i,j) = 0$. Then $r'$ is also the smallest degree of a hypersurface that contains $C$. The curve in the example above 
therefore does not lie on a quadric.
\end{enumerate}

There are more conditions that a diagram of an irreducible reduced space curve must satisfy, but we do not need them here.

Let us now enumerate all diagrams that an irreducible reduced space curve of degree $5$ with regularity $3$ can have if it does not lie on a quadric. First of all we have five $\infty$'s. Because $C$ has regularity $3$, there are at most $3$ rows that contain $\infty$'s. Since the entries decrease from top to bottom and from right to left, the only possible way to arrange $\infty$'s is as in the example above. The regularity $3$ also implies that the $4$th row contains only $0$'s. This leaves the entry at $(0,2)$. If $f(0,2)=0$, then $C$ lies on a quadric. If $f(0,2) = 1$, we have the example above. If $f(0,2) \ge 2$, we have that 
the Castelnuovo-Mumford regularity is $\ge 2+0+2 = 4$. So this leaves only the example above.
\end{proof}

We can sum up our discussion so far in

\begin{thm}
A nondegenerate reduced, irreducible space curve $C \subset \IP^3$ of degree $5$ that does not lie on
a quadric is either an elliptic quintic curve  or a smooth rational curve of degree $5$ with at least one $4$-secant.\qqed
\end{thm}

\begin{prop} \label{pRationalInvariant}
Let $C \subset \IP^3$ be a $\Group$-invariant 
smooth rational curve of degree $5$. Then, after a coordinate change, we have
\begin{align*}
	\phi \colon \IP^1 &\to \IP^3 \\
	(s:t) &\mapsto (\phi_1:\phi_2:\phi_3:\phi_4)
\end{align*}
with $\phi_i = as^5+bt^5$ with $a,b \not=0$ for one $i \in \{1,2,3,4\}$ and $\phi_j = c_jm_j$
with $m_j \in {st^4,s^2t^3,s^3t^2,s^4t}$ of weight $j-i$ for $j\not=i$.
\end{prop} 	

\begin{proof}
Since $C \cong \IP^1$ is $\Group$-invariant, $\Group$ also acts on $H^0(\ko_C(1))$; note that $\Group$ must act via projective automorphisms and the Schur multiplier of $\Group$ is trivial resp.\ it is easy to see in an elementary way that this action comes from a linear one. Let $s$
and $t$ be the characters of this action. Since $\phi$ is covariant, the $\phi_i$ span eigenlines for 
the $\Group$-action on degree $5$ polynomials in $s$ and $t$. Then the $\phi_i$ are linear combinations of monomials of the same weight. Since the weights of the coordinates of $\IP^3$ are $(1,2,3,4)$, the weight of $\phi_i$ has to be $k+i$ for some $k \in \Group$.
Therefore, it is not possible that $s$ and $t$ have the same weight. If the weights differ,
$s^5$ and $t^5$ have weight $0$ and the other weights occur for one monomial each.
Denote these monomials by $m_1,\dots,m_4$ and set $m_0 = as^5+bt^5$. Now $m_0$ has to occur as one of the $\phi_i$, since otherwise $st$ is a common divisor of all $\phi_i$ and the image of $\phi$ is a rational curve of degree $3$. The other $\phi_j$ must then be multiples of the remaining monomials. Also $a,b \not=0$, since otherwise either $t$ or $s$ is a common divisor of all $\phi_i$.
\end{proof}
 
\begin{cor}
There are no $\Group$-invariant rational curves of degree $5$ on the Fermat quintic $\Fermat$.
\end{cor}

\begin{proof}
If $\phi \colon \IP^1 \to C \subset \Fermat \subset \IP^3$ was invariant of degree $5$, we would have
\[
	\phi_1^5 + \phi_2^5 + \phi_3^5 + \phi_4^5 = 0 \in \IC[s,t]
\]
with $\phi_j$ as in Proposition \ref{pRationalInvariant}. In particular, 
\[
	\phi_i^5 = (as^5+bt^5)^5 = a^5s^{25} + \dots + b^5t^{25}
\]
with $a,b \not=0$.
The monomials $s^{25}$ and $t^{25}$ occur in none of the other $\phi_j^5$, $j\not=i$. Therefore, $\phi$ can never satisfy the Fermat equation.
\end{proof}

\begin{prop}
There are no non-degenerate irreducible reduced curves $C \subset \Fermat \subset \IP^3$ of degree $5$ that also lie on a quadric $Q$. 
\end{prop}

\begin{proof}
Let $D = \Fermat \cap Q$ be the complete intersection of the Fermat quintic with $Q$. If $C$ is a curve as above, $C$ is a reduced component of $D$.

If $Q$ is smooth, we have $Q \cong \IP^1 \times \IP^1$. In this case $D$ is a divisor of type $(5,5)$. Since $C$ has degree $5$, it can have types $(0,5)$, $(1,4)$ or $(2,3)$ (after possibly exchanging the two factors of $\IP^1 \times \IP^1$). A reduced divisor of type $(0,5)$ is the union of $5$ skew lines and therefore never irreducible. If $C$ has type $(1,4)$, then $C' = D-C$ has type $(4,1)$. We have $C.C' = 1^2 + 4^2 = 17$. Since $C$ is $\Group$-invariant, the scheme of intersection points is also $\Group$-invariant. In particular, its degree must be divisible by $5$ which is not the case. If $C$ has type $(2,3)$, then $C'$ has type $(3,2)$ and $C.C' = 2^2 + 3^2 = 13$ which is also not divisible by $5$. So there are no curves as in the proposition on a smooth quadric.

\newcommand{\Qtilde}{\widetilde{Q}}

If $Q$ is a quadric cone, it is defined by $\lambda x_i^2 + \mu x_jx_k$ with $2i=j+k \in \Group$. Therefore, the singular locus of $Q$ is a coordinate point and as such does not lie on the Fermat surface. Let $\Qtilde$ be the blowup of $Q$ in the singular point. Then $\Qtilde$ is a rational ruled surface. We use the notation of \cite[Sect.\ V.2]{Hartshorne}. Let $C_0$ be the exceptional divisor of the blowup. We have $C_0^2 = -2 =: -e$ (\cite[Ex.\ V.2.11.4]{Hartshorne}). Let, furthermore, $F$ be the strict transform of a line on $Q$. By \cite[Prop.\ V.2.3]{Hartshorne}, the Picard group of $\Qtilde$ is generated by $C_0$ and $F$ with $C_0.F = 1$ and $F^2=0$. The strict transform of a hyperplane section of $Q$ is
$C_1 = C_0+2F$. Let now $D = aC_0+bF$ be the strict transform of an invariant quintic curve $C$. Then 
\[
	5 = D.C_1 = (aC_0 + bF).(C_0+2F) = -2a + b + 2a = b.
\]
Now $C_0.D = C_0.(aC_0+5F) = -2a + 5$ is always nonzero, and $C$ must pass through the singular locus of $Q$. Therefore, $C$ cannot lie on the Fermat surface.

If $Q$ has rank $2$, then $D$ is the union of two plane quintics, which are degenerate.

If $Q$ has rank $1$, then $C$ must also be degenerate.
\end{proof}

Invariant elliptic curves of degree $5$ have been classified by Reid:

\begin{thm}[\cite{Reid}] \label{t:ellipticReid}
Let $E \subset \IP^3$ be a $\Group$-invariant elliptic quintic curve not containing any coordinate points. Then
\begin{itemize}
\item the homogeneous ideal of E is generated by $5$ cubics of the form
\begin{align*}
R_0 &= a {x}_{1}^{2} {x}_{3}-b {x}_{1} {x}_{2}^{2}+c {x}_{3}^{2} {x}_{4}-d {x}_{2} {x}_{4}^{2} \\
R_1 &= a s {x}_{1} {x}_{2} {x}_{3}-a t {x}_{1}^{2} {x}_{4}-b s {x}_{2}^{3}-c t {x}_{3} {x}_{4}^{2}\\
R_2 &= a s {x}_{1} {x}_{3}^{2}-b s {x}_{2}^{2} {x}_{3}-b t {x}_{1} {x}_{2} {x}_{4}-d t {x}_{4}^{3}\\
R_3 &= a t {x}_{1}^{3}+c s {x}_{2} {x}_{3}^{2}+c t {x}_{1} {x}_{3} {x}_{4}-d s {x}_{2}^{2} {x}_{4}\\
R_4 &= b t {x}_{1}^{2} {x}_{2}+c s {x}_{3}^{3}-d s {x}_{2} {x}_{3} {x}_{4}+d t {x}_{1} {x}_{4}^{2},
\end{align*}
where $a,b,c,d$ are nonzero constants and $(s:t) \in \IP^1$. For $E$ to be nonsingular, we must have $\frac{tbc}{sad} \not\in \left\{0,\infty,\frac{-11 \pm 5\sqrt{5}}{2} = \left( \frac{-1 \pm \sqrt{5}}{2}\right)^5 \right\}$. The set of all $E$ is parametrised $1$-to-$1$ by $(s:t) \in \IP^1$ and the ratio $(a:b:c:d) \in \IP^3$.
\item The vector space of $\Group$-invariant quintic forms vanishing on $E$ has a basis consisting of the 7 elements
\[
	x_1^2R_3, 
	x_2^2R_1,
	x_3^2R_4,
	x_4^2R_2,
	x_1x_4R_0,
	x_2x_3R_0,
	x_3x_4R_3.
\]
\end{itemize}
\end{thm}

\begin{proof}
Reid first considers elliptic normal curves in $\IP^4$. These are always defined by 
$4 \times 4$ Pfaffians of a linear skew-symmetric $5 \times 5$ matrix. He then shows that every invariant quintic elliptic curve in $\IP^3$ is the projection from $(1:0:0:0:0)$ of an elliptic normal curve whose defining matrix has normal form
\[
M(a:b:c:d,s: t) = 
\begin{pmatrix}0&
      x_{1}&
      x_{2}&
      x_{3}&
      x_{4}\\
      {-x_{1}}&
      0&
      c x_{3}&
      d x_{4}&
      s x_{0}\\
      {-x_{2}}&
      {-c x_{3}}&
      0&
      t x_{0}&
      a x_{1}\\
      {-x_{3}}&
      {-d x_{4}}&
      {-t x_{0}}&
      0&
      b x_{2}\\
      {-x_{4}}&
      {-s x_{0}}&
      {-a x_{1}}&
      {-b x_{2}}&
      0\\
      \end{pmatrix}.
\]
He then obtains the above formulas by eliminating $x_0$ from the Pfaffians of this matrix. See \cite{Reid} for further details. 
\end{proof}

\begin{prop}
The operation of $\mathrm{Aut}(X)$ on the parameter space $\IP^3 \times \IP^1$ of invariant quintic elliptic curves in $\IP^3$ is given by:
\begin{align*}
	\gamma(a:b:c:d,s:t) &= (\xi^2a:\xi b:\xi^{-1}c:\xi^{-2} d, s: t) \\
	\delta(a:b:c:d, s: t) &= (\xi a:\xi^{-1}b:\xi^{-1}c:\xi d,\xi^{-2} s:\xi^{2} t) \\
	\alpha (a:b:c:d, s: t) &= (-b:d:a:-c,t:-s) \\
	\beta(a:b:c:d,s:t) &= (d:c:b:a,s:t)
\end{align*}
\end{prop}

\begin{proof}
Apply $\gamma, \delta, \alpha$ and $\beta$ to $M(a:b:c:d,s:t)$ and renormalize.
\end{proof}

\begin{prop} \label{p:coeffElliptic}
There are exactly $50$ invariant elliptic quintic curves on the Fermat surface. They are
given by the following points in Reid's parameter space:
\[
	e_{i,j}^\pm = (\xi^{i+2j}:-\xi^{-i+j}:\xi^{-i-j}:-\xi^{i-2j}, \pm \golden^{\pm1}:\xi^{-i})
\]
with $\golden := - \xi^3- \xi^2$ the golden section. 
\end{prop}

\begin{proof}
The point $e_{0,0}^\pm$ has been calculated in \cite{GulPed}. The other points are obtained by
applying $\delta^j\gamma^i$. This shows that there are at least $50$ invariant elliptic quintic curves on the Fermat surface.

To show that there are exactly $50$ such curves we use Theorem \ref{t:ellipticReid}. By comparing coefficients of monomials in the $x_i$, we see that the condition that $x_1^5+x_2^5+x_3^5+x_4^5$ is in the given $7$-dimensional vector space is equivalent to 
\[
	\rank M := \rank 
	\begin{pmatrix}a t&
      0&
      0&
      0&
      0&
      0&
      0&
      1\\
      0&
      {-b s}&
      0&
      0&
      0&
      0&
      0&
      1\\
      0&
      a s&
      0&
      0&
      0&
      {-b}&
      0&
      0\\
      c s&
      0&
      b t&
      0&
      0&
      a&
      0&
      0\\
      0&
      0&
      c s&
      0&
      0&
      0&
      0&
      1\\
      {-d s}&
      {-a t}&
      0&
      0&
      {-b}&
      0&
      0&
      0\\
      c t&
      0&
      0&
      0&
      a&
      0&
      a t&
      0\\
      0&
      0&
      {-d s}&
      0&
      0&
      c&
      c s&
      0\\
      0&
      {-c t}&
      0&
      {-b s}&
      0&
      {-d}&
      {-d s}&
      0\\
      0&
      0&
      d t&
      a s&
      c&
      0&
      c t&
      0\\
      0&
      0&
      0&
      {-b t}&
      {-d}&
      0&
      0&
      0\\
      0&
      0&
      0&
      {-d t}&
      0&
      0&
      0&
      1\\
      \end{pmatrix} = 7.
\]
Since $a,b,c,d$ and $(s:t)$ are nonzero, the set of invariant elliptic curves on the Fermat is described in $\IP^3 \times \IP^1$  by the ideal of $8 \times 8$ minors of $M$ saturated by $a,b,c,d$ and $(s,t)$. With a computer algebra system one can check that the vanishing locus of this saturation has degree $50$. See \cite{BBS} for a Macaulay2 script doing this calculation.
\end{proof}

\begin{prop}\label{p:intersectionElliptic}
The intersection pairing on the $50$ degree $1$ elliptic curves on the Godeaux surface is
given by

\begin{center}
\begin{tabular}{c|ccc}
	$E^+_{i,j}.E^+_{k,l}$ & &  $j-l$ \\ 
		                        & 0 & $\pm1$ & $\pm2$ \\ \hline
	$i =k$		     & -1 & 0 & 1 \\
	$i\not=k$		     & 0 & 1 & 2 \\ 
\end{tabular} \quad\quad
\begin{tabular}{c|ccc}
	$E^-_{i,j}.E^-_{k,l}$ & &  $j-l$ \\ 
		                        & 0 & $\pm1$ & $\pm2$ \\ \hline
	$i =k$		     & -1 & 1 & 0 \\
	$i \not=k$		     & 0 & 2 & 1 \\ 
\end{tabular}
\end{center}
\vspace{5mm}
\begin{center}
\begin{tabular}{c|cc}
	$E^+_{i,j}.E^-_{k,l}$ & $j=l$ & $j\not=l$ \\ \hline
	$i=k$		     & 1 & 2 \\ 
	$i \not=k$		     & 0 & 1 \\
\end{tabular}
\end{center}
\end{prop}

\begin{proof}
By Reid's classification and Proposition \ref{p:coeffElliptic} we have the
explicit ideals for all elliptic curves. The degree of intersection on the Fermat surface is
the degree of the vanishing set of the sum of ideals. Dividing this degree by $5$
we obtain the intersection degree on the Godeaux surface. A Macaulay2 script
doing this calculation can be found at \cite{BBS}.
\end{proof}

\begin{proof}[Proof 2]
Since $\autGodeaux$ acts transitively on the set of elliptic quintic curves, we can assume
$E^\pm_{ij} = E^+_{00}$. Now the Galois group of $\IQ(\xi):\IQ$ is generated by
\[
	\rho_r(\xi) := \xi^r
\]
and the complex conjugation
\[
	\iota(\xi) := \bar{\xi}.
\]
We have
\[
	\rho_r(E^\pm_{kl}) = \left\{
	\begin{matrix}
	E^\pm_{rk,rl} & \text{if $r=1,3$} \\
	E^\mp_{rk,rl} & \text{if $r=2,4$} 
	\end{matrix}
	\right.
\]
and
\[
	\iota(E^\pm_{k,l}) = E^\pm_{-k,-l}.
\]
In particular, the subgroup $\{\rho_1,\iota\rho_3,\rho_3,\iota\}$ leaves $E^+_{00}$ invariant.
Applying these Galois transformations we can assume that $E^\pm_{kl}$ is either
$E^\pm_{0l}$ or $E^\pm_{1l}$ depending on whether $k=0$ or $k\not=0$. Now
\[
	\beta(E^\pm_{kl}) = E^\pm_{k,-l}.
\]
So we can assume that $l = 0,1,2$. Calculating the intersections $E^+_{00}.E^+_{kl}$ for $k=0,1$ and $l=0,1,2$ using a computer algebra program we obtain the first table. Applying $\iota \circ \rho_4$ gives the second table.

If $E^+_{00}.E^-_{kl}$ with $k=0,1$ and $l=0,1,2$, we consider the operation of $\alpha$
\[
	\alpha (E^\pm_{k,l}) = E^\mp_{-k,-2l}.
\]
If $l=2$ we obtain
\begin{align*}
	&\beta \delta \gamma^4\alpha (E^+_{00}.E^-_{k2}) \\
	&=\beta \delta^k \gamma^4( E^-_{00}.E^+_{-k,-4})\\
	&=\beta(E^-_{k4}.E^+_{00})\\
	&=E^-_{k1}.E^+_{00}
\end{align*}
We therefore only need to compute $E^+_{00}.E^-_{kl}$ with $k,l \in \{0,1\}$, which yields the third table.
\end{proof}

\begin{prop}\label{p:intersectionEllipticLine}
The intersection pairing of the $50$ degree $1$ elliptic curves with the $15$ lines on the Godeaux surface is
given by

\begin{center}
\begin{tabular}{c|ccc}
	 & &  $i-k$ \\ 
		                                & 0 & $\pm1$ & $\pm2$ \\ \hline
	$E^\pm_{i,j}.L^\pm_k$ & 1 & 2 & 0 \\
	$E^\pm_{i,j}.L^\mp_k$ & 1 & 0 & 2 \\ 
\end{tabular} \quad\quad
\begin{tabular}{c|ccc}
	 & &  $j-k$ \\ 
		                       & 0 & $\pm1$ & $\pm2$ \\ \hline
	$E^-_{i,j}.L^0_k$ & 3 & 0 & 1 \\
	$E^+_{i,j}.L^0_k$ & 3 & 1 & 0 \\
\end{tabular} \quad\quad
\end{center}
\end{prop}

\begin{proof}
This can again be done by direct calculation. A Macaulay2 script doing this can be found at \cite{BBS}. The symmetry is explained by the operation of $\autGodeaux$ and $\Gal(\IQ(\xi):\IQ)$.
\end{proof}

We can now write down an integral basis of $N(X)$ consisting of classes of \emph{effective} divisors. This facilitates many computations in the following.

\begin{cor}
The effective divisors \[ \{E^+_{10},E^+_{20},E^+_{30},E^+_{40},K\otimes \mathcal{O}_{\tau},E^-_{11},E^-_{22},E^-_{33},E^-_{44} \} ,\] where $\mathcal{O}_{\tau}$ is a nontrivial torsion bundle on $X$, form a $\IZ$-basis of $N(X)$, the group of divisors on $X$ modulo numerical equivalence. The intersection matrix is given by
\[
	\left(
	\begin{tabular}{cccc|c|cccc}
	-1 & 0 & 0 & 0  & 1 & 2 & 1 & 1 & 1 \\
	0 & -1 & 0 & 0 & 1 & 1 & 2 & 1 & 1 \\ 
	 0 & 0 & -1 & 0  & 1& 1 & 1 & 2 & 1 \\ 
	 0 & 0 & 0 & -1 & 1 & 1 & 1 & 1 & 2 \\ \hline
	 1 & 1 & 1 & 1  & 1& 1 & 1 & 1 & 1 \\ \hline
	 2 & 1 & 1 & 1 & 1 & -1 & 2 & 1 & 1 \\ 
	 1 & 2 & 1 & 1  & 1& 2 & -1 & 2 & 1 \\ 
	 1 & 1 & 2 & 1  & 1& 1 & 2 & -1 & 2 \\ 
	 1 & 1 & 1 & 2  & 1 & 1 & 1 & 2 & -1 
	\end{tabular}
	\right) .
\]

\end{cor}

\begin{proof}
All elliptic curves constructed above are of degree $1$, i.e $E^\pm_{ij}.K=1$. The other
intersection numbers follow from Proposition \ref{p:intersectionElliptic}. The determinant of this matrix is $1$, so we have a $\IZ$-Basis.
\end{proof}

In Section \ref{sE8} we will exhibit another $\mathbb{Z}$-basis of $N(X)$, with respect to which the intersection pairing becomes that of the $\textbf{1}\oplus (-E_8)$-lattice, but it does not consist of effective divisors and is therefore somewhat less well adapted to computations.

\begin{remark}\xlabel{rGuletskiiPedrini}
In \cite{GulPed}, Guletskii and Pedrini write down a potential $\mathbb{Q}$-basis of $N(X)_{\mathbb{Q}}$. This contains an error, however, as the second author kindly confirmed: the intersection matrix in \cite{GulPed} is not correct.  They write $M_2.M_3 =1$, whereas the correct intersection number is $5$, compare our Proposition \ref{p:intersectionLine}. With this correction, their intersection matrix does not have full rank. \end{remark}

\section{The $E_8$-symmetry}\xlabel{sE8}

The orthogonal complement of $K_X$ in $N(X)$, the group of divisors on the Godeaux surface $X$ modulo numerical equivalence, is the lattice of type $-E_8$. The symmetry group is the Weyl group $W(E_8)$ of the root system of type $E_8$. Its order is $2^{14} \times 3^5 \times 5^2 \times 7$. 
It has a description in terms of finite groups of Lie type as an extension
\[
1\to \mathbb{Z}/2\mathbb{Z} \to W(E_8) \to \mathrm{O}_8^+ (\mathbb{F}_2) \to 1
\]
where $\mathrm{O}_8^+ (\mathbb{F}_2)$ is one of the orthogonal groups over $\mathbb{F}_2$ (see \cite{MT}). This description can be found in \cite{Bourbaki}.

We have to introduce some notation below to be able to talk about this. Recall the fundamental (see e.g. \cite{Humphreys}, \cite{Bourbaki}, \cite{FH})

\begin{definition}\xlabel{dRootSystem}
Let $E$ be a Euclidean vector space, i.e. a finite dimensional real vector space with a symmetric positive definite bilinear form $( \cdot , \cdot )$. For $\alpha \neq 0$ we use the standard abbreviation
\[
\langle \beta , \alpha \rangle := \frac{2 (\beta , \alpha )}{(\alpha , \alpha )} \, .
\]
The reflection $\sigma_{\alpha }$ in the hyperplane orthogonal to $\alpha \neq 0$ is then given by
\[
\sigma_{\alpha } (\beta ) = \beta - \langle \beta , \: \alpha \rangle \alpha \, .
\]
A subset $\Phi$ of $E$ is called a root system if the following axioms are satisfied:
\begin{itemize}
\item[(R1)]
$\Phi$ is finite, spans $E$ and does not contain $0$.
\item[(R2)]
If $\alpha \in \Phi$, then the only multiples of $\alpha$ in $\Phi$ are $\pm \alpha$.
\item[(R3)]
If $\alpha$ is in $\Phi$, then the reflection $\sigma_{\alpha }$ leaves $\Phi$ invariant.
\item[(R4)]
If $\alpha$, $\beta\in\Phi$, then $\langle \beta , \alpha \rangle \in \mathbb{Z}$.
\end{itemize}
\end{definition} 

The terminology comes from the fact that if $\mathfrak{g}$ is a complex semisimple Lie algebra, $\mathfrak{h} \subset \mathfrak{g}$ a Cartan algebra, and $\mathfrak{h}^*$ is endowed with the Killing bilinear form, then the nonzero weights $\Phi$ of the adjoint representation of $\mathfrak{g}$ form a root system.

In the following we use freely basic results about root systems and adhere to the standard terminology in e.g. \cite{Humphreys}. In particular, each root system defines a lattice $R \subset E$ generated by $\Phi$, the root lattice. Recall that a subset $\Delta$ of $\Phi$ is a basis of the root system if it is a basis of $E$ and each root can be written as an integral linear combination of elements of $\Delta$ with all coefficients nonnegative or nonpositive. The roots in $\Delta$ are then called simple roots. If $\{ \alpha_1, \dots , \alpha_l\}$ is a system of simple roots, the matrix $(\langle \alpha_i, \alpha_j \rangle )_{i,j}$ is the Cartan matrix of $\Phi$. It is not necessarily symmetric. It determines the root system up to isomorphism in the following sense: if $\Phi' \subset E'$ with base $\Delta' = \{ \alpha_1', \dots , \alpha_l' \}$ is another root system, and $\langle \alpha_i, \: \alpha_j \rangle = \langle \alpha_i', \alpha_j' \rangle$ for all $i, j$, then the bijection $\alpha\mapsto \alpha_i'$ extends uniquely to an isomorphism $\varphi$ of $E$ and $E'$ mapping $\Phi$ onto $\Phi'$ and with $\langle \varphi ( \alpha ), \varphi (\beta ) \rangle = \langle \alpha , \beta \rangle $ for $\alpha , \beta \in \Phi$.

The Coxeter graph of $\Phi$ is a graph having $l$ vertices, the $i$th joined to the $j$th ($i\neq j$) by $\langle \alpha_i, \alpha_j \rangle \langle \alpha_j, \alpha_i \rangle$ edges. It allows one to recover the Cartan matrix if all roots have the same lengths. Otherwise one passes to the Dynkin diagram by introducing an arrow pointing to the shorter of the two roots $\alpha_i$ or $\alpha_j$.

The Weyl group $W$ is the subgroup of the isometries of $E$ generated by all the reflections $\sigma_{\alpha }$.

The root system of type $E_8$ can then be described as follows (following \cite{Humphreys}, which is the notation most widely used). Take $E = \mathbb{R}^8$ with its standard inner product and standard basis $e_1, \dots , e_8$. Let $\Phi$ consist of the vectors
\[
\pm (e_i \pm e_j), \quad \frac{1}{2}\sum_{i=1}^8 (-1)^{k(i)} e_i
\]
where $k(i)=0$ or $1$ and the $k(i)$ add up to an even integer. If we choose as an ordered basis 

\begin{gather*}
\Delta= ( \alpha_1, \dots , \alpha_8)= \\
(\frac{1}{2}(e_1 + e_8 -(e_2+\dots + e_7)),  e_1+e_2,  e_2 -e_1, e_3-e_2, \\
  e_4-e_3,   e_5-e_4,  e_6 -e_5,  e_7-e_6 )\, ,
\end{gather*}
then the Cartan matrix becomes
\[
\left( \begin{array}{cccccccc}
2 & 0 & -1 & 0 & 0 & 0 & 0 & 0\\
0 & 2 & 0 & -1 & 0 & 0 & 0 & 0\\
-1 & 0 & 2 & -1& 0 & 0 & 0 & 0\\
0 & -1 & -1 & 2 & -1 & 0 & 0 & 0\\
0 & 0 & 0 & -1 & 2 & -1 & 0 & 0\\
0 & 0 & 0 & 0 & -1 & 2 & -1 & 0 \\
0 & 0 & 0 & 0 & 0 & -1 & 2 & -1 \\
0 & 0 & 0 & 0 & 0 & 0 & -1 & 2 
\end{array}\right)
\]
Here we have $(\alpha , \alpha ) =2$ for all roots. So the preceding matrix also gives the scalar product $(\cdot , \cdot )$ in the basis $\{ \alpha_1, \dots , \alpha_8 \}$. It agrees with the standard form given in \cite[Ex.\ I.2.7]{BPV}. Note that here $\alpha_2$ corresponds to the distinguished node in the Dynkin diagram.

The root lattice above endowed with the symmetric negative definite bilinear form given by minus the preceding Cartan matrix is called the $-E_8$-lattice. Somewhat by abuse of notation we will also talk of roots, simple roots etc. in this setting.

Having described the relevant lattices, we need one more general notion before we pass to geometry. 

\begin{definition}\xlabel{dSubsystem}
Let $\Phi$ be a root system. A subset $\Psi \subset \Phi$ is said to be closed if 
\begin{itemize}
\item[(C1)]
for all $\alpha$, $\beta \in \Psi$ we have $\sigma_{\alpha} (\beta ) \in \Psi$
\item[(C2)]
for $\alpha, \beta \in \Psi$ with $\alpha + \beta \in \Phi$, we have $\alpha + \beta \in \Psi$.
\end{itemize}
\end{definition}

See e.g. \cite{MT}, p. 104 for a discussion of this fundamental notion. It is clear that closed subsets are root systems in their own right in the appropriate Euclidean spaces. The algorithm of Borel-Siebenthal determines all closed subsystems of $\Phi$, see \cite{MT}, p. 108 ff. It proceeds as follows: form the extended Dynkin diagrams as in Table 13.1, p. 108 of \cite{MT}. For any proper subdiagram of the extended Dynkin diagram form the extended Dynkin diagram of each indecomposable part of that subdiagram and repeat the process. At any stage of the process, the set of nodes of the current diagram is a subset of $\Phi$ (the extra node corresponding to the highest root) and gives a basis of a closed subsystem. In fact all maximal closed subsystems are obtained in this way up to conjugation by $W$.

We now pass to geometry. Let $S$ be a degree $1$ del Pezzo surface, i.e. the blow-up of $\mathbb{P}^2$ in $8$ points in general position. The orthogonal to $-K_S$ in $N(S)$ is a lattice of type $-E_8$ as is the complement of $K_X$ in $N(X)$. We want to describe the roots in this geometric context. We denote by $k$ the canonical class in $N(S)$, by $r_1, \dots , r_8$ the classes of the $8$ exceptional divisors, and by $h$ the hyperplane class. Then
\[
-k = 3h - r_1 - \dots - r_8
\]
and putting
\begin{gather*}
\alpha_1 = r_1-r_2, \quad \alpha_2 = h - r_1-r_2-r_3, \quad \alpha_3 = r_2-r_3, \quad \alpha_i =r_{i-1}  -r_i, \; i\ge 4 \, 
\end{gather*}
we get simple roots with matrix the negative of the Cartan matrix above. Hence we have recovered the simple roots in this picture. 

\begin{prop}\xlabel{pGodeauxRoots}
For the Godeaux surface a set of simple roots can be identified as
\begin{gather*}
\alpha_1 = E^-_{0,4}-E^+_{4,4}, \\
\alpha_2 = E_{1,4}^+ -E_{2,4}^-, \\
\alpha_3 = E_{4,0}^+ -E_{3,0}^+, \\
\alpha_4 = E_{3,0}^+ -E_{2,0}^+, \\
\alpha_5 = E_{2,0}^+ -E_{1,0}^+, \\
\alpha_6= E_{1,0}^+ - E_{0,0}^-, \\
\alpha_7 = E_{0,2}^- - E_{0,4}^-, \\
\alpha_8= E_{0,3}^- - E_{0,0}^- .
\end{gather*}
\end{prop}

\begin{proof}
One checks directly that the intersection matrix with respect to these vectors becomes the standard one of $-E_8$. See \cite{BBS} for a Macaulay 2 script doing this.
\end{proof}

\begin{remark}\xlabel{rHomeo}
The presence of $-E_8$ in both cases has a deeper geometric reason: $X$ is homeomorphic to the connected sum of a degree one del Pezzo surface with a rational homology sphere $\Sigma^4$, see \cite{HK}, p. 87. The latter is responsible for the torsion fundamental group. A Barlow surface, which is simply connected, is known to be homeomorphic to a degree one del Pezzo surface itself. 
\end{remark}

We produce a   numerically exceptional sequence of length $11$ in $N(X) \simeq \mathbf{1}\perp (-E_8)$. By this we mean a sequence 
\[
l_1, \dots , l_{11}
\]
of classes $l_i\in N(X)$ of line bundles on $X$ such that $\chi (l_j, l_i) =0$ for $j>i$. Note that because by Riemann-Roch
\[
\chi (l_j, l_i) = 0 \iff  (l_i-l_j)( l_i-l_j -K_X) = -2
\]
this is a purely lattice theoretic consideration. We would like to emphasize throughout the connection to the $E_8$-symmetry. We denote the generator of $\mathbf{1}$ again by $k$.

\begin{prop}\xlabel{pExtended}
Consider the vectors
\begin{gather*}
\alpha_9 = K - E^+_{0,2}, \\
\alpha_{10} = E^-_{0,4} - E^-_{0,1}
\end{gather*}
Then also $\alpha_{10},  \alpha_9, \alpha_8, \alpha_7, \alpha_6, \alpha_5, \alpha_4, \alpha_3$ is an ordered basis of the root system of type $E_8$ with respect to which the intersection matrix is of the standard type $-E_8$ above (so $\alpha_9$ is the distinguished node in this numbering).
\end{prop}

\begin{proof}
This can be checked by Macaulay 2, see \cite{BBS}.
\end{proof}

Moreover,
$\alpha_{10}, \alpha_8, \alpha_7, \alpha_6, \alpha_5, \alpha_4, \alpha_3, \alpha_1$ is a basis in a subsystem of type $A_8$. It is therefore reasonable to introduce the following new notation to do justice to the situation:
\begin{gather*}
A_1 =\alpha_1, A_2= \alpha_3, A_3 = \alpha_4, A_4=\alpha_5, A_5 = \alpha_6, A_6 = \alpha_7, A_7 = \alpha_8, A_8= \alpha_{10}, \\
B_1 = \alpha_2 , B_2= \alpha_{9}
\end{gather*}
and to represent this graphically as follows:
\vspace{1cm}
\begin{center}
\setlength{\unitlength}{1cm}
\begin{picture}(8,1)
\put(0,0){$\bullet$}
\put(1,0){$\bullet$}
\put(2,0){$\bullet$}
\put(3,0){$\bullet$}
\put(4,0){$\bullet$}
\put(5,0){$\bullet$}
\put(6,0){$\bullet$}
\put(7,0){$\bullet$}

\put(2,1){$\bullet$}
\put(5,1){$\bullet$}

\put(0.1,0.1){\line(1,0){1}}
\put(1.1,0.1){\line(1,0){1}}
\put(2.1,0.1){\line(1,0){1}}
\put(3.1,0.1){\line(1,0){1}}
\put(4.1,0.1){\line(1,0){1}}
\put(5.1,0.1){\line(1,0){1}}
\put(6.1,0.1){\line(1,0){1}}

\put(2.1,0){\line(0,1){1}}
\put(5.1,0){\line(0,1){1}}

\put(0,-0.5){$A_1$}
\put(1,-0.5){$A_2$}
\put(2,-0.5){$A_3$}
\put(3,-0.5){$A_4$}
\put(4,-0.5){$A_5$}
\put(5,-0.5){$A_6$}
\put(6,-0.5){$A_7$}
\put(7,-0.5){$A_8$}

\put(1.5, 1){$B_1$}
\put(5.2,1){$B_2$}

\dashline{0.2}(2.1,1.1)(5.1,1.1)
\put(3.4,1.3){$-1$}
\end{picture}
\end{center}
\vspace{1cm}

Note that here every node has self-intersection $-2$ and different nodes have intersection zero unless they are joined by an edge in which case their intersection is $1$: the only exception is $B_1$ and $B_2$ where it may be checked directly that $B_1\cdot B_2 = -1$.

\begin{prop}\xlabel{propNumerics}
The sequence
\begin{align*}
A_1,\\
A_1+A_2,\\
k - B_1,\\
A_1+A_2+A_3, \\
A_1+A_2+A_3+A_4, \\
A_1+A_2+A_3+A_4+A_5, \\
k-B_2,\\
A_1+A_2+A_3+A_4+A_5+A_6, \\
A_1+A_2+A_3+A_4+A_5+A_6+A_7, \\
A_1+A_2+A_3+A_4+A_5+A_6+A_7+A_8, \\
\mathcal{O}
\end{align*}
is numerically exceptional of length $11$. 
\end{prop}

In fact, the exceptional sequence we construct later - when twisted by $\mathcal{O}_X (-E_{4,4}^+ + E_{0,4}^-)$ - has exactly this numerical behaviour. 

\begin{proof}
Note that $k\cdot A_i = k\cdot B_j =0$ for all $i$ and $j$. For a class $m$ with $m\cdot k =0$ the condition $\chi (\mathcal{O}, m) =0$ is equivalent to $m^2 = -2$. Moreover, $\chi (k-B_i) = \chi (B_i)$. Hence we see that $\mathcal{O}$ is orthogonal to all the members of the sequence preceding it because the $A_{i_0} + A_{i_0+1} + \dots + A_{i_0 +t}$ are all roots and so are the $B_i$. 

Now the sequence 
\[
A_1, A_1+A_2, \dots , A_1 + \dots + A_8
\]
is completely orthogonal: this follows again because all the $A_{i_0} + A_{i_0+1} + \dots + A_{i_0 +t}$ and their negatives are roots. 

Hence it suffices to show that $k-B_1$ is numerically semi-orthogonal to all the terms preceding it, and that the same is true for $k-B_2$. We have for $j\le 5$ 
\begin{gather*}
\chi (k-B_2, A_1+ \dots +A_j) = \\
\frac{1}{2}(A_1+ \dots + A_j + B_2 -k)\cdot (A_1+ \dots + A_j + B_2 -2k) +1 =0
\end{gather*}
and the calculation for $k-B_1$ is analogous.
Finally, we have $\chi (k-B_2, k-B_1) = \chi (B_2 -B_1) = 0$, too.
\end{proof}

Note that the Weyl group acts on the set of numerically exceptional sequences of a given length. 

\section{Campedelli Model}\xlabel{sTower}

The Campedelli construction of the Godeaux surface $X$ roughly consists in realizing $X$ as the double cover of $\mathbb{P}^1\times \mathbb{P}^1$ branched in a union of five fibers of the second projection $\mathbb{P}^1\times\mathbb{P}^1 \to \mathbb{P}^1$ and a curve of bidegree $(6,7)$ which has $10$ triple points on the preceding fibers, $2$ on each fiber (one still has to desingularize the resulting cover and contract some $(-1)$-curves arising as the strict transforms of the fibers afterwards). Extensive background can be found in \cite{Reid}, p.\ 312 ff. Here we will review the aspects of this construction relevant for us later, and complement it with some extra facts which we will need. We start with the explicit description of the set of $10$ points in $\mathbb{P}^1\times \mathbb{P}^1$.
\smallskip

Following Reid \cite{Reid}, p.\ 315 ff., we consider the following
$\Group$-invariant polynomials.

\begin{align*}
	\phi_1 &= x_1^2x_3+x_4^2x_2\\
	\phi_2 &= x_1x_2^2+x_4x_3^2\\
	\psi_1 &= x_1^3x_2-x_4^3x_3\\
	\psi_2 &= x_1x_3^3-x_4x_2^3
\end{align*}

Here $\beta$ operates with character $+1$ on the $\phi_i$ and with character $-1$ on the
$\psi_i$. 
These polynomials represent sections
\begin{align*}
	H^0(3K-L_0^0) &= \langle \phi_1,\phi_2 \rangle \\
	H^0(4K-L_0^0) &= \langle \psi_1,\psi_2 \rangle
\end{align*}
on the Godeaux surface and define a generically $2:1$ rational map
\[
	(\psi,\phi) \colon \Godeaux \dashrightarrow \IP^1 \times \IP^1
\]
and $\beta$ exchanges the preimages of a general point in $\IP^1 \times \IP^1$.
The
subgroup $\betacom \subset \autGodeaux$ that commutes with $\beta$ is generated by $\delta$ and $\alpha$.
We obtain the following operation on $\IP^1 \times  \IP^1$
\begin{align*}
	\delta(\phi_1:\phi_2) &= (\xi \phi_1:\xi^4 \phi_2) \\
	\delta(\psi_1:\psi_2) &= (\xi^2 \psi_1:\xi^3 \psi_2) \\
	\alpha(\phi_1:\phi_2) &= (\phi_2:\phi_1) \\
	\alpha(\psi_1:\psi_2) &= (\psi_2:- \psi_1).
\end{align*}

Consider  the 10 points on $\IP^1 \times \IP^1$ that are defined by the following ideals
\[
	I^\pm_a = (\psi_1\mp \golden^{\pm1} \xi^{a}\psi_2,\phi_1+\xi^{3a}\phi_2)
\]
with $a \in \IZ / 5\IZ$ and $\golden := - \xi^3- \xi^2$ the golden section.
The operation of $\betacom$ on this set of ideals is transitive and given by
\begin{align*}
	\delta(I^\pm_a) &= I^\pm_{a+1} \\
	\alpha (I^\pm_a) &= I^\mp_{-a}
\end{align*}
We also denote by $p^\pm_a = V(I^\pm_a) \in \IP^1 \times \IP^1$ the points defined
by the ideals above.

\begin{prop}
The preimage of $p^\pm_i$ under $(\psi,\phi)$ is $E^\pm_{i,0}$, which we denote by $E^\pm_i$ in this section.
\end{prop}

\begin{proof}
On can check, see \cite{BBS}, for $p^+_0$ that the hypersurfaces $\phi_1+\phi_2=0$
and $\psi_1- \golden \psi_2=0$ cut out the preimage of $L_0^0$ and $E^+_{0,0}$. The rest follows because
$B$ acts transitively as indicated above.
\end{proof}

We now explain a tower of maps and a geometric set-up that facilitates certain computations with the Campedelli model. We will give some examples of this below which will be of relevance later.

We introduce the following notation. (See the following picture.)

\begin{enumerate}

\item[(A)]
Let $S:= \mathbb{P}^1 \times\mathbb{P}^1$ be the Hirzebruch surface with projections $\pi_1$ and $\pi_2$ to the first and second factors, and let $C$, $F_1, \dots ,F_5$ be the arrangement of divisors in $S$ described in \cite{Reid}; thus $F_1, \dots ,F_5$ are five fibers of the projection $\pi_2$ and $C$ is a curve of bidegree $(6,7)$ having two triple points $p_i^+$ and $p_i^-$ on each fiber $F_i$.
\item[(B)]
Let $\hat{\sigma}\colon \hat{S} \to S$ be the blow-up of $S$ in the ten points $p_i^+$, $p_i^-$ and let $\hat{P}_i^+ = \hat{\sigma}^{-1} (p_i^+)$, $\hat{P}_i^- = \hat{\sigma }^{-1}(p_i^-)$ be the exceptional divisors. Let $\hat{F}_i\subset \hat{S}$ be the strict transform of $F_i$, and let $\hat{C}\subset \hat{S}$ be the strict transform of $C$. Then $\hat{C}$ intersects each of $\hat{P}_i^+$ and $\hat{P}_i^-$ transversely in three points, and $\hat{F}_i \simeq \mathbb{P}^1$, $\hat{F}_i^2 = -2$ because
$(\hat{F}_i + \hat{P}_i^+ +\hat{P}_i^-)^2 = F_i^2 = 0$.
\item[(C)]
Let $\tilde{X}$ be the double cover of $\hat{S}$ ramified in $\hat{C} \cup \hat{F}_1 \cup \dots \cup \hat{F}_5$ with map $\eta\colon \tilde{X} \to \hat{S}$. The preimages of $\hat{P}^+_i$ resp. $\hat{P}^-_i$ in $\tilde{X}$ are elliptic curves $\tilde{E}_i^+$ resp. $\tilde{E}_i^-$ joined together by the preimage $\tilde{F}_i$ of $\hat{F}_i$ which is a $(-1)$-curve on $\tilde{X}$. Here, $\tilde{F}_i$ is the reduced divisor associated to $\eta^{-1}(\hat{F}_i)$, hence, because the covering is branched, we have equality of $\mathbb{Q}$-divisors $\eta^* ((1/2) \hat{F}_i) = \tilde{F}_i$, and $\tilde{F}_i^2 = (1/4)\cdot \hat{F}_i^2 \cdot \mathrm{deg} (\eta ) =-1$. We have $(\tilde{E}_i^{\pm })^2 = -2$. Indeed, $\eta^* (\hat{P}^{\pm }_i) = \tilde{E}_i^{\pm }$ and $\tilde{E}_i^2 = \deg (\eta ) \cdot (\hat{P}_i^{\pm })^2= -2$.
\item[(D)]
Let $\tilde{\sigma } \, :\, \tilde{X} \to X$ be the contraction of the five $(-1)$-curves $\tilde{F}_1, \dots ,\tilde{F}_5$ to points $q_1, \dots ,q_5$ on $X$. The images of the $\tilde{E}_i^+$ resp. $\tilde{E}_i^-$ on $X$ are denoted $E_i^+$ resp. $E_i^-$ and are elliptic curves meeting in a point $q_i$. Here we have $(E_i^{\pm })^2 = -1$, for $(E_i^{\pm })^2= \sigma^* (E_i^{\pm})^2 = (\tilde{E}_i^{\pm} + \tilde{F}_i)^2 = -2 +2\cdot 1 -1 = -1$. Then $X$ is the Godeaux surface. 
\end{enumerate}
\smallskip

\begin{center}
\includegraphics[height=90mm]{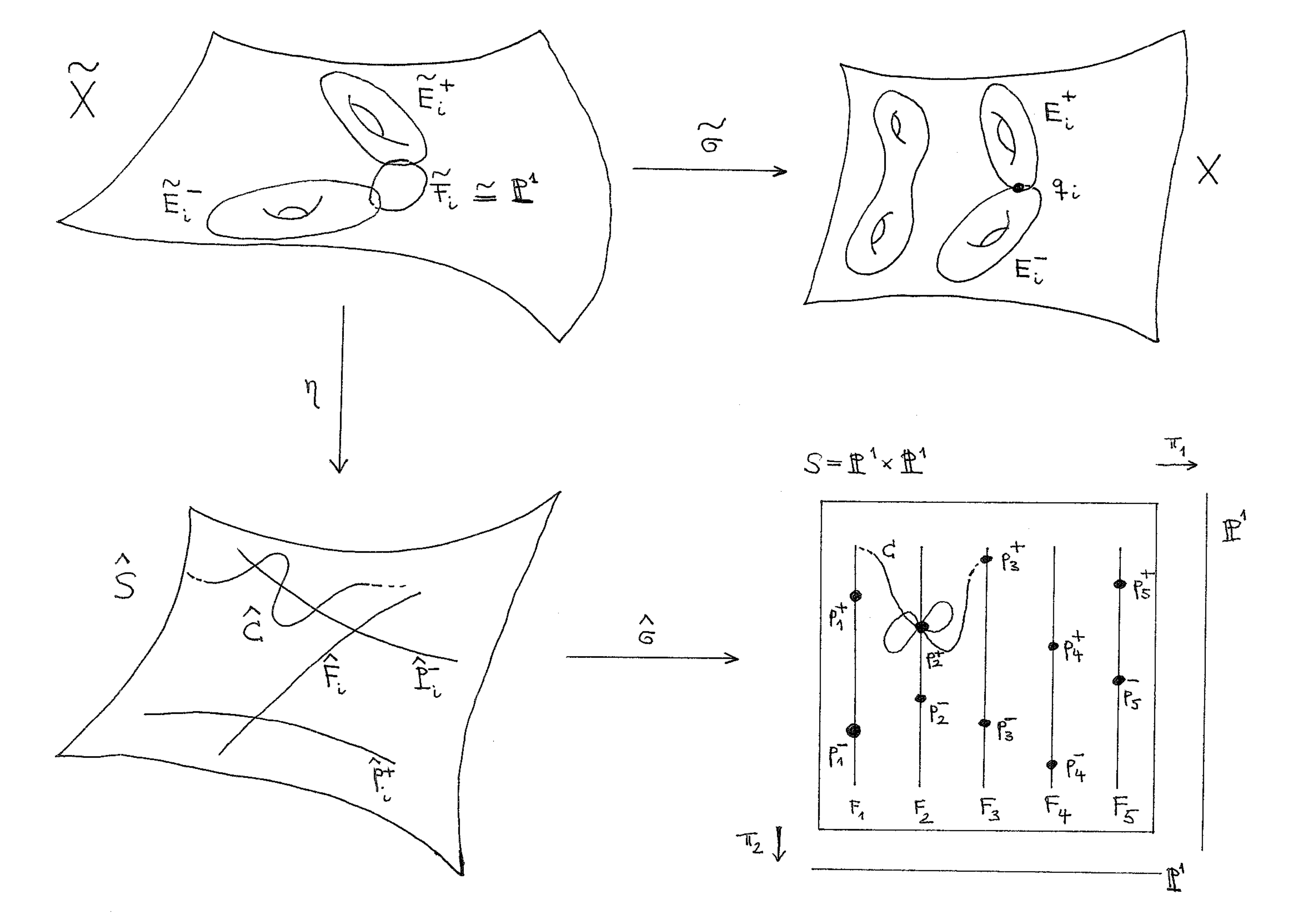}
\end{center}
\smallskip

Next we list some useful formulas pertaining to this set-up.

\begin{lem}
\begin{itemize}
\item[(1)]
We have
\[
p_{\ast } \mathcal{O}_Y = \mathcal{O}_X \oplus \bigoplus_{\tau \in (\Group-\{ 0\})} \mathcal{O}_{\tau }, 
\]
where $\tau$ runs over the nontrivial elements in $\mathrm{Pic}(X)_{\mathrm{tors}}$ and $\mathcal{O}_{\tau}$ is a representative of the isomorphism class of $\tau$.
\item[(2)]
The direct image $\eta_{\ast}\mathcal{O}_{\tilde{X}}$ is calculated as
\[
 \eta_{\ast}\mathcal{O}_{\tilde{X}} = \mathcal{O}_{\hat{S}} \oplus \mathcal{O}_{\hat{S}}(-1/2(\hat{C} + \sum \hat{F}_i))    .
\]
\item[(3)]
For the respective canonical classes we have
\[
K_{\tilde{X}} = \eta^* \left( K_{\hat{S}} + \frac{1}{2}(\hat{C} + \sum \hat{F}_i)\right)     
\]
and
\[
K_{\hat{S}} + \frac{1}{2} (\hat{C} + \sum \hat{F}_i )= \hat{\sigma }^* \left( K_{S} + \frac{1}{2}(C + \sum F_i )  \right) - \sum (\hat{P}^+_i + \hat{P}^-_i) \, .
\]
\item[(4)]
Finally, what will sometimes be useful in connection with the projection formula is that
\[
\tilde{\sigma }_{\ast} \left( \mathcal{O}_{\tilde{X}}(\sum n_k\tilde{F}_k)\right) = \mathcal{O}_X   
\]
if all $n_k \ge 0$. It shows that $H^0 (\tilde{X}, \tilde{\sigma }^*\mathcal{L}) = H^0 (\tilde{X}, \tilde{\sigma }^*\mathcal{L} \otimes \mathcal{O}_{\tilde{X}}(\sum n_k\tilde{F}_k))$ for a line bundle $\mathcal{L}$ on $X$. 
\end{itemize}
\end{lem}

\begin{proof}
The second formula in (3) follows because for the blow-up we have
\[
\hat{\sigma }^* (K_S) = K_{\hat{S}} - \sum (\hat{P}_i^+ +\hat{P}_i^-)
\]
and, moreover,
\[
\hat{\sigma }^* \left(  \frac{1}{2} (C + \sum F_i)  \right) = \frac{1}{2} \left( \sum \hat{F}_i  + \hat{C} + 4\sum (\hat{P}^+_i+\hat{P}^-_i) \right)  
\]
because the multiplicity of $C + \sum F_i$ is $4$ in each of the points $p_i^+$, $p_i^-$. The remaining formulas are standard and can be found in \cite{BPV}. For instance, (2) and the first formula in (3) are in \cite[Sect.\ V.22]{BPV}. 
\end{proof}

We establish some vanishing results used later on.

\begin{prop}\xlabel{pGeneral}
Put
\[
D = k K_X + \sum_{i} e_i^+ E_i^+ +\sum_i e_i^- E_i^-
\]
with $k, e_i^+, e_i^-\in\mathbb{Z}$ and 
\[
2\ge k\ge 0, \quad e_i^++e_i^- - k\le 0, \quad e_i^+-k\le 0, \quad e_i^- -k\le 0 \;\forall i\, .
\]
Assume that
\[
H^0 \left( \mathbb{P}^1\times\mathbb{P}^1, \mathcal{O}_{\mathbb{P}^1\times\mathbb{P}^1}(k,4k)  \otimes \bigoplus_i \left( \mathcal{I}_{p_i^+}^{k-e_i^+} \oplus \mathcal{I}_{p_i^+}^{k-e_i^-} \right)                 \right) =0
\]
or, in words, that there are no nonzero effective divisors of bidegree $(k,4k)$ on $\mathbb{P}^1\times\mathbb{P}^1$ passing with multiplicity $k-e_i^+$ through $p_i^+$ and with multiplicity $k-e_i^-$ through $p_i^-$. Then
\[
H^0 (X,\mathcal{O}_X (D))= 0\, .
\] 
\end{prop}

\begin{proof}
\begin{align*}
&H^0 (X,k K_X + \sum_{i} e_i^+ E_i^+ +\sum_i e_i^- E_i^-) \\
 &= H^0 (\tilde{X}, \tilde{\sigma }^*( k K_X + \sum_{i} e_i^+ E_i^+ +\sum_i e_i^- E_i^-))\\
 &= H^0 (\tilde{X}, \mathcal{O}_{\tilde{X}} (kK_{\tilde{X}} +(e_i^++e_i^- - k)(\sum \tilde{F}_i )+\sum_{i} e_i^+ \tilde{E}_i^+ +\sum_i e_i^- \tilde{E}_i^- ))\\
 &=  H^0 (\tilde{X}, \mathcal{O}_{\tilde{X}} (kK_{\tilde{X}}  +  \sum_{i} e_i^+ \tilde{E}_i^+ +\sum_i e_i^- \tilde{E}_i^- ))\\
 &= H^0 (\tilde{X}, \mathcal{O}_{\tilde {X}} (\eta^{\ast} (k K_{\hat{S}} +\frac{k}{2} ( \hat{C} + \sum \hat{F}_i)) +  \sum_{i} e_i^+ \tilde{E}_i^+ +\sum_i e_i^- \tilde{E}_i^- )) \\
 &= H^0 \left( \hat{S}, \mathcal{O}_{\hat{S}}\left(  k K_{\hat{S}} +\frac{k}{2} ( \hat{C} + \sum \hat{F}_i)  +   \sum_{i} e_i^+ \hat{P}_i^+ +\sum_i e_i^- \hat{P}_i^-    \right)  \right) \\
&\oplus H^0 \left( \hat{S}, \mathcal{O}_{\hat{S}}\left(  k K_{\hat{S}} +\frac{k-1}{2} ( \hat{C} + \sum \hat{F}_i)  +   \sum_{i} e_i^+ \hat{P}_i^+ +\sum_i e_i^- \hat{P}_i^-    \right)  \right) .
\end{align*}
We investigate the first term of the previous direct sum first:
\begin{align*}
& H^0 \left( \hat{S}, \mathcal{O}_{\hat{S}}\left(  k K_{\hat{S}} +\frac{k}{2} ( \hat{C} + \sum \hat{F}_i)  +   \sum_{i} e_i^+ \hat{P}_i^+ +\sum_i e_i^- \hat{P}_i^-    \right)  \right)\\
&=H^0 \left( \hat{S}, \mathcal{O}_{\hat{S}}\left(  \hat{\sigma }^* \left( kK_{S} + \frac{k}{2}(C + \sum F_i )  \right) - \sum k(\hat{P}^+_i + \hat{P}^-_i)+   \sum_{i} e_i^+ \hat{P}_i^+ +\sum_i e_i^- \hat{P}_i^-    \right)  \right)\\
 &=H^0 \left( \mathbb{P}^1\times\mathbb{P}^1, \mathcal{O}_{\mathbb{P}^1\times\mathbb{P}^1}(k,4k)  \otimes \bigoplus_i \left( \mathcal{I}_{p_i^+}^{k-e_i^+} \oplus \mathcal{I}_{p_i^+}^{k-e_i^-} \right)                 \right) .
\end{align*}
Since 
\[
\hat{\sigma }^* (-\frac{1}{2} (C+ \sum F_i) ) = -\frac{1}{2} (\hat{C} + \sum \hat{F}_i ) - 2 \sum (\hat{P}_i^+ + \hat{P}_i^-)
\]
and $\mathcal{O} (-\frac{1}{2} (C+ \sum F_i)) = \mathcal{O}_{\PP^1\times \PP^1} (-3,-6)$, an analogous computation shows that 
\[
H^0 \left( \hat{S}, \mathcal{O}_{\hat{S}}\left(  k K_{\hat{S}} +\frac{k-1}{2} ( \hat{C} + \sum \hat{F}_i)  +   \sum_{i} e_i^+ \hat{P}_i^+ +\sum_i e_i^- \hat{P}_i^-    \right)  \right)
\]
has no sections because $k\le 2$. 
The claim follows. 
\end{proof}

\begin{cor}\xlabel{cGeneral}
The following hold:
\begin{itemize}
\item[(1)]
$H^0 (X, K_X)=0$.
\item[(2)]
$H^0 (X, K_X + E^+_{i_0} - E^+_{j_0})=0$ for $i_0 \neq j_0$.
\item[(3)]
Let $I, J, K, L$ be four pairwise different indices out of $\{ 1, \dots ,5\}$. Then we have
\[
H^0 (X, 2K_X +E^{\pm}_I +E^{\pm}_J -E_K^{\pm} -E^{\pm}_L) = 0 \, .
\]
(Here arbitrary combinations of $+$ and $-$ are allowed, so the only requirement is that the elliptic curves are from distinct fibres).
\end{itemize}
\end{cor}

\begin{proof}
We apply Proposition \ref{pGeneral}.

(1): Sections in $H^0(X, K_X)$ correspond to effective divisors of bidegree $(1,4)$ through the points $p_i^+$, $p_i^-$ (it follows that there are no such nontrivial divisors).

(2): We have to consider divisors $D$ of bidegree $(1,4)$ on $S$ passing through all the points $p_i^+$, $p_i^-$ with the possible exception of $p_{i_0}^+$, and passing doubly through the point $p_{j_0}^+$. But a nontrivial such $D$ would have to contain as components all the four lines which $p_{i_0}^+$ does not lie on. Hence $D$ minus these four lines would be a divisor of bidegree $(1,0)$ which has to pass simply through $p_{j_0}^+$ and through $p_{i_0}^-$. This is impossible because the points $p_i^+$, $p_i^-$ all project to different points in $\mathbb{P}^1$ under $\pi_1$. 

(3): We have to consider effective divisors $D$ of bidegree $(2,8)$ on $\mathbb{P}^1\times \mathbb{P}^1$ which pass simply through $p_I^{\pm}$, $p_J^{\pm}$, triply through $p_K^{\pm}$, $p_L^{\pm}$, and doubly through all the remaining points. If $D$ is a nontrivial such divisor, it has to contain $F_K$ and $F_L$ doubly and the other fibers simply. Hence $D-2 (F_K+F_L)-$(remaining fibers) is a divisor of bidegree $(2,1)$ through six points as illustrated in the right hand picture below. 
\smallskip

\begin{center}
\setlength{\unitlength}{1cm}
\begin{picture}(10, 2)

\put(-1, 2){\line(0,-1){2}}
\put(-0.5, 2){\line(0,-1){2}}
\put(0, 2){\line(0,-1){2}}
\put(0.5,2){\line(0,-1){2}}
\put(1,2){\line(0,-1){2}}

\put(1.5,1){(2,8)-divisors}
\put(1.5,0.5){ through this}

\put(4,1){\vector(1,0){1}}

\put(6, 2){\line(0,-1){2}}
\put(6.5, 2){\line(0,-1){2}}
\put(7, 2){\line(0,-1){2}}
\put(7.5,2){\line(0,-1){2}}
\put(8,2){\line(0,-1){2}}

\put(8.5,1){(2,1)-divisors}
\put(8.5,0.5){ through this}

\put(-1.2,1.5){$\times$}
\put(-1.2, 0.5){$\times$}
\put(-1.05,0.5){$\times$}

\put(-0.7,1.5){$\times$}
\put(-0.7, 0.5){$\times$}
\put(-0.55, 0.5){$\times$}

\put(-0.2,1.5){$\times$}
\put(-0.05, 1.5){$\times$}
\put(-0.12, 1.65){$\times$}
\put(-0.2, 0.5){$\times$}
\put(-0.05, 0.5){$\times$}

\put(0.3,1.5){$\times$}
\put(0.45, 1.5){$\times$}
\put(0.38, 1.65){$\times$}
\put(0.3, 0.5){$\times$}
\put(0.45, 0.5){$\times$}

\put(0.8,1.5){$\times$}
\put(0.95, 1.5){$\times$}
\put(0.8, 0.5){$\times$}
\put(0.95, 0.5){$\times$}

\put(5.8,1.5){$\circ$}
\put(5.8, 0.5){$\times$}

\put(6.3,1.5){$\circ$}
\put(6.3, 0.5){$\times$}

\put(6.8,1.5){$\times$}
\put(6.8, 0.5){$\circ$}

\put(7.3,1.5){$\times$}
\put(7.3, 0.5){$\circ$}

\put(7.8,1.5){$\times$}
\put(7.8, 0.5){$\times$}

\end{picture}
\end{center}

It can be checked by computer \cite{BBS} that all of these point configurations here impose independent conditions on divisors of bidegree $(2,1)$ and hence there are no such nontrivial divisors. 
\end{proof}

\section{The magic of the torsion}\xlabel{sTorsion}

It will turn out that the torsion in the Picard group of the Godeaux surface can perform some magic for us which will be indispensable for the construction of the length $11$ exceptional sequence, namely, it can charm away unwanted sections. By this we mean that although $H^0 (X,\mathcal{L}) \neq 0$ for some line bundle $\mathcal{L}$ on $X$, we will have $H^0 (X,\mathcal{L}\otimes \mathcal{O}_{\tau}) = 0$ for $\tau \in \mathrm{Pic}(X)_\mathrm{tors}$ nontrivial. We now briefly discuss these phenomena.

\begin{lem}\xlabel{lTorsion}\xlabel{lTorsionElliptic}
Let $\mathcal{L}= \mathcal{O}(D)$ be an effective line bundle on $X$ with $\mathcal{L}^2 < 0$ and $D$, $p^* (D)$ irreducible divisors. Then for $\tau \in \mathrm{Pic}(X)_\mathrm{tors}\simeq \Group$ nontrivial, one has
\[
H^0 (X,\mathcal{L} \otimes \mathcal{O}_{\tau}) = 0 \, .
\]
\end{lem}

\begin{proof}
It suffices to note that for the cover $p\colon Y \to X$ (where $Y$ is the Fermat) we have $(p^* \mathcal{L})^2 < 0$, hence $H^0 (Y,p^*(\mathcal{L}))=1$ and $H^0 (X,\mathcal{L})=1$. Since by the projection formula
\[
H^0 (Y, p^*(\mathcal{L})) = H^0 (X,\mathcal{L}) \oplus \bigoplus_{\tau \in \Group-\{ 0 \}} H^0 (X,\mathcal{L} \otimes \mathcal{O}_{\tau }),
\]
the claim follows.
\end{proof}

The previous lemma is basically all we will use. We remark however that it may also very well happen that $H^0 (X, \mathcal{L}) = 0$ for a line bundle $\mathcal{L}$ on $X$, whereas $H^0 (X, \mathcal{L}\otimes \mathcal{O}_{\tau }) \neq 0$: this happens for example for $\mathcal{L} = K_X$ where $H^0 (X, \mathcal{L}\otimes \mathcal{O}_{\tau })= 1$ (see \cite{Reid}). We will use the following obvious result to ensure later that we do not create unwanted sections in $\mathrm{Hom}$-spaces of our line bundle sequence. 

\begin{lem}\xlabel{lObvious}
If $\mathcal{L}$ is not numerically equivalent to any effective divisor, then always
\[
H^0 (X,\mathcal{L} \otimes \mathcal{O}_{\tau }) = 0 . 
\]
\qqed
\end{lem}

\begin{lem} \xlabel{l:TorsionEL}
Let $I, J, K, L$ be four pairwise different indices out of $\{ 0, \dots ,4\}$ and consider
\[
	D := 2K_X +E^{\pm}_I +E^{\pm}_J -E_K^{\pm} -E^{\pm}_L
\]
(Here arbitrary combinations of $+$ and $-$ are allowed, so the only requirement is that the elliptic curves are from distinct fibres). Then we have
\[
	H^0(\mathcal{O}(D) \otimes \mathcal{O}_{\tau^i}) = 
	\left\{
	\begin{tabular}{cl}
	      1 & \text{if $D = K + L^\pm_k$ and $i\not=0$} \\
	      0 & \text{otherwise.}
	\end{tabular}
	\right.
\]
\end{lem}

\begin{proof}
We use the linear equivalence
\[
	3K-L^0_0 = E_{i}^+ + E_{i}^-
\]
to write 
\[
	D = 8K_X - 2 L^0_0 - E^{\mp}_I - E^{\mp}_J -E_K^{\pm} -E^{\pm}_L.
\]
Since $I,J,K,L \in \{0,\dots,4\}$ are different there is exactly one missing index.
After repeated application of $\delta \in \autGodeaux$ we can assume that the missing index
is $0$. So we only need to consider
\[
	D = 8K_X - 2 L^0_0 - E^{\pm}_1 - E^{\pm}_2 -E_3^{\pm} -E^{\pm}_4.
\]
A comparison of numerical classes shows that of the $16$ sign combinations only
$+--+$ and $-++-$ are of the type $D = K + L^\pm_k$. We now compute the ideal of the
curve 
$$D' = 2 L^0_0  + E^{\pm}_1 +  E^{\pm}_2 + E_3^{\pm}  + E^{\pm}_4$$
on the Fermat surface in all $16$ cases using a Macaulay2 script \cite{BBS}. It turns out that in 
the two cases above there is (modulo the Fermat equation) a unique degree $7$ polynomial $F$ that
contains $D'$. One can further check that $F$ is $\Group$ invariant. It follows that
in degree $8$ we have the equations $Fx_1,\dots, Fx_4$ on which $\Group$ acts with weights $1,2,3,4$. In particular none of them ist $\Group$ invariant. This shows that
$\mathcal{O}(D) = \mathcal{O}(8K-D')$ has no section and $\mathcal{O}(D)\otimes \mathcal{O}_{\tau}^i$ has one section for each $i \not=0$.

In the other $14$ cases the ideals are generated (except for the Fermat equation) in degree $9$. Therefore 
\[
	H^0(\mathcal{O}(D)\otimes \mathcal{O}_{\tau^i}) = 0
\]
for all $i$ if $D \not= K + L^\pm_k$.
\end{proof}

\section{The exceptional sequence of length $11$ on $X$}\xlabel{sExceptional}

Let $\Db (X)$ be the bounded derived category of coherent sheaves on $X$. Recall the

\begin{definition}\xlabel{dExceptional}
An object $\mathcal{E}$ in $\Db (X)$ is called exceptional if $\mathrm{RHom}^{\bullet} (\mathcal{E}, \mathcal{E}) \simeq \mathbb{C}$. A sequence $(\mathcal{E}_1, \dots ,\mathcal{E}_n)$ of exceptional objects $\mathcal{E}_i$ is called an exceptional sequence if $\mathrm{RHom}^{\bullet} (\mathcal{E}_i, \mathcal{E}_j) =0$ whenever $i>j$. The sequence is called complete if the smallest full triangulated subcategory containing all the $\mathcal{E}_i$ is equivalent to $\Db (X)$. 
\end{definition}

It is clear that a variety with a complete exceptional sequence has a free Grothendieck group, and that the maximal length of an exceptional sequence is always bounded by the rank of that group. It follows that for our Godeaux surface $X$ this length is bounded by $11$. In fact, this is attained. For this we lift the sequence described in Proposition \ref{propNumerics} after tensoring it with $\mathcal{O}_X(E^+_{4,4} - E^-_{0,4})$:

\begin{thm}\xlabel{tMain}
Let $X$ be the classical Godeaux surface and $\tau\in\Group$ a nontrivial torsion class. Then there is an exceptional sequence of length $11$ in $\Db(X)$:
\[
(\mathcal{L}_1, \dots ,\mathcal{L}_{11})
\]
where all the $\mathcal{L}_i$ are line bundles (of course, the sequence is not complete). In the notation established in Section \ref{sDegree1}, these $\mathcal{L}_i$ are the following:
\begin{align*}
\mathcal{L}_1 &= \mathcal{O}_X(E^+_{4, 0}-E^+_{4,0}) =\mathcal{O}_X  \, , \\
\mathcal{L}_2 &=  \mathcal{O}_X(E^+_{4, 0}-E^+_{3,0})  \, , \\
\mathcal{L}_3 &=: \mathcal{M} = \mathcal{O}_X( K_X + E^+_{0,0} - E^+_{1,0} - E^+_{2,0} + E^+_{4,0} ) \, , \\
\mathcal{L}_4 &=  \mathcal{O}_X( E^+_{4,0} - E^+_{2,0} ) \, , \\
\mathcal{L}_5 &=  \mathcal{O}_X(E^+_{4,0} - E^+_{1,0}) \, , \\
\mathcal{L}_6 &= \mathcal{O}_X (E^+_{4,0} - E^-_{0,0})  \, , \\
\mathcal{L}_7 &=: \mathcal{N} =  \mathcal{O}_X(K_X + E^+_{0,0} - E^+_{1,0} - E^+_{2,0} + E^-_{3,0}) \, , \\
\mathcal{L}_8 &=\mathcal{N}\otimes \mathcal{O}(E_{0, 3}^+)^{-1}\otimes \mathcal{O}_{\tau } \simeq_{\mathrm{num}} \mathcal{O}_X(E^+_{4,1} - E^-_{0,1})  \, , \\
\mathcal{L}_9 &= \mathcal{O}_X(E^+_{4,2} - E^-_{0,2}) \, , \\
\mathcal{L}_{10} &= \mathcal{O}_X(E^+_{4,3} - E^-_{0,3})  \, , \\
\mathcal{L}_{11} &=\mathcal{N}\otimes \mathcal{O}(E_{0, 2}^+)^{-1}\otimes \mathcal{O}_{\tau } \simeq_{\mathrm{num}} \mathcal{O}_X(E^+_{4,4} - E^-_{0,4} ) \, .
\end{align*}

For $i \neq 3,7$ the line bundle $\mathcal{L}_i$ is of degree $0$, while $\mathcal{L}_3 = \mathcal{M}$ and $\mathcal{L}_7 = \mathcal{N}$ are of degree $1$. 
\end{thm}

\begin{proof} One checks (for example with Macaulay2 \cite{BBS}) that the classes of $\mathcal{L}_i$ tensored with $\mathcal{O}_X(E^+_{4,4} - E^-_{0,4})^*$ agree with those of Proposition \ref{propNumerics}.

Since $h^1 (X,\mathcal{O}_X ) = h^2 (X,\mathcal{O}_X ) = 0$, every line bundle on $X$ is exceptional. 
First of all one then has that $\chi (\mathcal{L}_i,\mathcal{L}_j) = 0$ for all $i > j$. This follows from Proposition \ref{propNumerics}. Thus it suffices to show that 
\[
h^0 (X,\mathcal{L}_i^{\vee} \otimes \mathcal{L}_j ) = 0 \;\;\; \mathrm{and}\;\;\; h^2 (X,\mathcal{L}_i^{\vee} \otimes \mathcal{L}_j)= h^0 (X,\mathcal{L}_j^{\vee} \otimes \mathcal{L}_i \otimes K_X) =0 \; \mathrm{for} \; \mathrm{all} \; i>j \, .
\]
Now the bundles $\mathcal{L}_i^{\vee} \otimes \mathcal{L}_j$ and $\mathcal{L}_j^{\vee} \otimes \mathcal{L}_i \otimes K_X$ are of degree $-1$, $0$ or $1$ in all but the following seven cases:
\[
\mathcal{L}_j^{\vee} \otimes \mathcal{L}_i \otimes K_X \; \mathrm{for} \; (i,j) = (3,2), \: (3,1), \: (7,6), \: (7,5), \: (7,4), \: (7,2), \: (7,1) \, .
\]
In these cases $\mathcal{L}_j^{\vee} \otimes \mathcal{L}_i \otimes K_X$ has degree $2$. Let us call these the exceptional cases and exclude them for the moment. In the non-exceptional cases one can check (by computer \cite{BBS}) that the respective line bundles are not even numerically equivalent to any effective divisor (here we use the classification of the effective degree $1$ divisors in Section \ref{sDegree1}) unless we look at $\mathcal{L}_{11}^{\vee} \otimes \mathcal{L}_7 = \mathcal{L}_{11}^{\vee}\otimes \mathcal{N}$ or $\mathcal{L}_{8}^{\vee} \otimes \mathcal{L}_7 = \mathcal{L}_{8}^{\vee}\otimes \mathcal{N}$. These are numerically equivalent to elliptic curves on $X$ of self-intersection $-1$, however, the way we chose $\mathcal{L}_8$ and $\mathcal{L}_{11}$ above shows that up to \emph{linear equivalence} the bundles $\mathcal{L}_{11}^{\vee}\otimes \mathcal{N}$ and $\mathcal{L}_{8}^{\vee}\otimes \mathcal{N}$ are really of the form $\mathcal{O}_X (F) \otimes \mathcal{O}_{\tau }$ where $F$ is such an elliptic curve and $\tau$ is a nontrivial torsion class in $\Group$.  Hence they do not have sections by Lemma \ref{lTorsionElliptic}. 
\smallskip

Let us now consider the remaining exceptional cases. It turns out that in all these cases $\mathcal{L}_j^{\vee} \otimes \mathcal{L}_i \otimes K_X$ is (up to \emph{linear equivalence} resp. isomorphism of line bundles) equal to
\[
	 2K_X +E^{\pm}_I +E^{\pm}_J -E_K^{\pm} -E^{\pm}_L
\]
where $I, J, K, L$ are four pairwise different indices out of $\{ 1, \dots , 5\}$. Hence we get the desired vanishing of global sections in these cases by Corollary \ref{cGeneral}. \end{proof}

\begin{remark}\xlabel{rNonvanishing}
This gives a counterexample to Kuznetsov's Nonvanishing Conjecture, see \cite[Conj.\ 9.1]{Kuz09}. In fact, we have a semiorthogonal decomposition in $\Db(X)$
\[
\Db(X) = \left\langle \mathcal{A},\mathcal{L}_1, \dots ,\mathcal{L}_{11},  \right\rangle
\]
where $\mathcal{A}$ is the admissible subcategory (for this notion and that of a semiorthogonal decomposition we refer to \cite[Sect.\ 2]{Kuz09}) which is right orthogonal to the exceptional sequence of line bundles (hence it must be written to the left of the sequence; go figure). The Hochschild homology  $\mathrm{HH}_{\bullet}(\mathcal{A})$ is zero: Hochschild homology is additive on semiorthogonal decompositions and by the Hochschild-Kostant-Rosenberg isomorphism 
\[
\mathrm{HH}_{i} (\mathrm{D}^b (X))\simeq \bigoplus_{q-p=i} H^p (X, \Omega^q_X)\, .
\]
Hence, since all rational cohomology classes on $X$ are algebraic, $\mathrm{HH}_{i} (\mathrm{D}^b (X))\simeq \mathbb{C}^{11}$. But 
\[
\mathrm{HH}_{\bullet} (\langle \mathcal{L}_1, \dots , \mathcal{L}_{11} \rangle ) \simeq \mathbb{C}^{11} ,
\]
thus $\mathrm{HH}_{\bullet } (\mathcal{A}) =0$. However, it is not true that $\mathcal{A} = 0$ as predicted by the conjecture: the exceptional sequence is not full because the Grothendieck group of $X$ is not free. 
\end{remark}

\section{Explicit objects in the orthogonal to the exceptional sequence}\xlabel{sOrthogonal}

In the preceding section we saw that there is a semi-orthogonal decomposition

\[
\Db(X) = \left\langle \mathcal{A}, \mathcal{L}_1, \dots ,\mathcal{L}_{11}\right\rangle,
\]
where $\mathcal{A} \neq 0$ and does not contain any further exceptional objects. In fact, all objects in $\mathcal{A}$ have vanishing Chern character. Here we produce some explicit nontrivial objects in $\mathcal{A}$.  We will later need the following

\begin{lem}\xlabel{lOther}
Let $\mathcal{O}_{\tau}$ be a nontrivial torsion bundle on $X$. Then the sequence
\[
(\mathcal{O}_{\tau },\mathcal{L}_2 , \dots ,\mathcal{L}_{11})
\]
is also exceptional.
\end{lem}
\begin{proof} By Lemma \ref{lObvious}  and Lemma \ref{l:TorsionEL}
it suffices to show that the line bundles $\mathcal{O}_X (K_X) \otimes \mathcal{L}_i$ for $i=3$, $7$  which are of the form 
\[
\mathcal{O}_X ( 2K_X +E^{\pm}_I +E^{\pm}_J -E_K^{\pm} -E^{\pm}_L ) 
\]
are not numerically equivalent to $K+L^\pm_k$. This is again checked in \cite{BBS}.
\end{proof}

Concretely, notice that for any nontrivial torsion bundle $\mathcal{O}_{\tau}$ on $X$ we have 
\[
\mathrm{Ext}^2 (\mathcal{O}_X,\mathcal{O}_{\tau }) \simeq H^0 (X,K_X \otimes \mathcal{O}_{\tau }^{-1 })^{\vee } \simeq \mathbb{C}
\]
hence we get a nontrivial arrow in $\Db(X)$:
\[
\mathcal{O}_X \to \mathcal{O}_{\tau } [2] \, .
\]
Let $C_{\tau }$ be the mapping cone of this. Note that $C_{\tau}$ is nontrivial and its class in the Grothendieck group $\mathrm{K} (X)$ is $5$-torsion: indeed, for the Fermat cover $p\colon Y \to X$ we have in $\mathrm{K} (X)$
\[
5([\mathcal{O}_X] - [\mathcal{O}_{\tau }]) = p_{\ast} p^{\ast} ([\mathcal{O}_X] - [\mathcal{O}_{\tau }]) = p_{\ast}([\mathcal{O}_Y] - [\mathcal{O}_Y]) = 0\, .
\]

Then we have:

\begin{prop}\xlabel{pOrthogonal}
The object $C_{\tau }$ is nontrivial and in $\mathcal{A}$, i.e.\ $\mathrm{RHom}^{\bullet} (\mathcal{L}_i,C_{\tau }) = 0$ for all $i=1,\dots ,11$.
\end{prop}

\begin{proof}
From the distinguished triangle
\[
\mathcal{O}_X \to \mathcal{O}_{\tau}[2] \to C_{\tau } \to \mathcal{O}_X [1]
\]
we get a long exact sequence
\[
\dots \to \mathrm{Ext}^k (\mathcal{O}_X ,\mathcal{O}_X) \to \mathrm{Ext}^k (\mathcal{O}_X,\mathcal{O}_{\tau }[2]) \to \mathrm{Ext}^k (\mathcal{O}_X,C_{\tau }) \to \dots \, .
\]
As $\mathrm{Ext}^k (\mathcal{O}_X ,\mathcal{O}_X) = H^k (X,\mathcal{O}_X ) =0$ for $k\neq 0$ and $\mathrm{Ext}^k (\mathcal{O}_X,\mathcal{O}_{\tau }[2])= H^{k+2}(X,\mathcal{O}_{\tau })$ (note also that
\[
H^0 (X,\mathcal{O}_{\tau }) =0, \: H^1 (X,\mathcal{O}_{\tau } ) = 0, \: H^2 (X,\mathcal{O}_{\tau }) = \mathbb{C} \, \mathrm{)}
\]
we get
\[
 0 \to \mathrm{Ext}^{-1}(\mathcal{O}_X,C_{\tau }) \to H^0 (X,\mathcal{O}_X) \to H^2 (X,\mathcal{O}_{\tau }) \to \mathrm{Ext}^{0}(\mathcal{O}_X,C_{\tau }) \to 0
\]
and $\mathrm{Ext}^k (\mathcal{O}_X,C_{\tau } ) =0$ for $k \neq -1,0$. The map between $H^0(X,\mathcal{O}_X)\simeq \mathrm{Hom} (\mathcal{O}_X\mathcal{O}_X)$ and $H^2 (X,\mathcal{O}_{\tau }) \simeq \mathrm{Hom} (\mathcal{O}_X,\mathcal{O}_{\tau}[2])$ (both one-dimensional spaces) is given by the composition with the nontrivial arrow $\mathcal{O} \to \mathcal{O}_{\tau } [2]$, hence is an isomorphism. Thus $\mathrm{Ext}^k (\mathcal{O}_X,C_{\tau } ) =0$ for all $k$. Notice that $\mathcal{L}_1 = \mathcal{O}_X$.
\smallskip

Consider now one of the $\mathcal{L}_i$ with $i=2, \dots ,11$. Again we get a long exact sequence
\[
\dots \to \mathrm{Ext}^k (\mathcal{L}_i ,\mathcal{O}_X) \to \mathrm{Ext}^k(\mathcal{L}_i,\mathcal{O}_{\tau }[2]) \to \mathrm{Ext}^k (\mathcal{L}_i, C_{\tau }) \to \dots \, .
\]
But because the original sequence is exceptional, $\mathrm{Ext}^k (\mathcal{L}_i,\mathcal{O}_X ) =0$ for all $i\geq 2$ and all $k$, and likewise, by Lemma \ref{lOther}, $\mathrm{Ext}^k (\mathcal{L}_i, \mathcal{O}_{\tau}[2] ) \simeq \mathrm{Ext}^{k+2}( \mathcal{L}_i \otimes \mathcal{O}_{\tau}^{-1}, \mathcal{O}_X)=0$ for all $i\geq 2$ and all $k$. Hence, $\mathrm{Ext}^k (\mathcal{L}_i,C_{\tau }) = 0$ for all $i$ and $k$, and $C_{\tau}$ is in $\mathcal{A}$ as claimed.
\end{proof}

\section{The $A_{\infty}$-algebra of derived endomorphisms of the exceptional sequence}\xlabel{sRigidity}

Consider a deformation $\{ X_t \}$ of $X = X_0$ among numerical Godeaux surfaces with $\ZZ/5$-torsion, all of whose canonical models are $\ZZ/5$-quotients of quintics in $\PP^3$. Since, $p_g=q=0$, the sequence of line bundles $(\mathcal{L}_i)$ deforms along with the surface, and by upper-semicontinuity,
\[
(\mathcal{L}_{1, t}, \dots , \mathcal{L}_{11, t} )
\]
is still an exceptional sequence for $t$ in a small neighbourhood of $0$ in the family. Moreover, since Bloch's conjecture holds for all these surfaces by Voisin's work \cite{Voisin}, there is a decomposition
\[
\mathrm{D}^b (X_t ) = \langle \mathcal{A}_t, \mathcal{L}_{1, t}, \dots , \mathcal{L}_{11, t} \rangle
\]
with $\mathrm{K}_0 (\mathcal{A}_t ) = \ZZ/5$. 

Consider $\mathbb{L}_t = \bigoplus_{i=1}^{11} \mathcal{L}_{i, t}$ and the differential graded algebra $\mathfrak{A}_t = \mathrm{RHom}^{\bullet} (\mathbb{L}_t, \mathbb{L}_t)$ of derived endomorphisms of the exceptional sequence above. It has a minimal model in the sense of \cite{Keller01}, 3.3, i.e. we consider the Yoneda algebra $H^* (\mathfrak{A}_t)$ together with its $A_{\infty}$-structure such that $m_1=0$, $m_2=$Yoneda multiplication and there is a quasi-isomorphism of $A_{\infty}$-algebras $\mathfrak{A}_t \simeq H^* (\mathfrak{A}_t)$ lifting the identity of $H^* (\mathfrak{A}_t)$. The goal of this section is to show 

\begin{prop}\xlabel{pRigidity}
The $A_{\infty}$-algebra $H^* (\mathfrak{A}_t)$ is constant in a neighbourhood of a generic point of the family $\{ X_t \}$. The categories
\[
\langle \mathcal{L}_{1, t}, \dots , \mathcal{L}_{11, t} \rangle
\]
are all equivalent in that neighbourhood. 
\end{prop} 

The argument is inspired by the proof of rigidity for the exceptional sequences similar to ours in the recent preprint \cite{A-O12} by Alexeev and Orlov. We also would like to thank Dmitri Orlov for explaining the argument to us and helping with the proof of the statement in Proposition \ref{pRigidity}. 

We use

\begin{lem}\xlabel{lRigidity}
The following hold:
\begin{itemize}
\item[(1)]
The sequence \[ (\mathcal{L}_1, \mathcal{L}_2, \mathcal{L}_4, \mathcal{L}_5, \mathcal{L}_6, \mathcal{L}_8, \mathcal{L}_9, \mathcal{L}_{10}, \mathcal{L}_{11})\] is completely orthogonal. Moreover, $(\mathcal{L}_3, \mathcal{L}_7)$ are completely orthogonal. 
\item[(2)]
We have
\begin{gather*}
\chi (\mathcal{L}_i, \mathcal{L}_3 ) = -1,\;  i < 3, \quad \chi (\mathcal{L}_j, \mathcal{L}_7) = -1,\;   j < 7, j \neq 3 \\
\chi (\mathcal{L}_3, \mathcal{L}_s ) = 1, \; 3 < s, s\neq 7, \quad \chi (\mathcal{L}_7, \mathcal{L}_t ) = 1, \; 7 < t \, .
\end{gather*}
\item[(3)]
One has 
\[
\mathrm{Hom} (\mathcal{L}_i, \mathcal{L}_j) \neq 0 \iff i=j \; \mathrm{or }\;  (i,j) = (2,3) \, .
\]
\item[(4)]
If $\chi (\mathcal{L}_i, \mathcal{L}_j ) =1$, then $\mathrm{Ext}^2 (\mathcal{L}_i, \mathcal{L}_j) = \mathbb{C}$. 
\item[(5)]
If $\chi (\mathcal{L}_i, \mathcal{L}_j ) = -1$, then $\mathrm{Ext}^2 (\mathcal{L}_i, \mathcal{L}_j) = (0)$.
\end{itemize}
\end{lem}

\begin{proof}
Most of it is checked already in \cite{BBS}. The only additional statement that we needed to check with Macaulay 2 is (4). The rest is obvious by direct computation.
\end{proof}

Note that the same properties will hold for the sequence $(\mathcal{L}_{i,t})$ by upper-semicontinuity for $t$ in a small neighbourhood of $0$, except possibly  (3): it may happen that both $\dim \mathrm{Hom}(\mathcal{L}_{2,t}, \mathcal{L}_{3,t})$ and $\dim \mathrm{Ext}^1 (\mathcal{L}_{2,t}, \mathcal{L}_{3,t})$ go down by $1$ away from $t=0$. This is irrelevant for the subsequent argument. 
 
Let us recall now some facts about $A_{\infty}$-categories that we need to prove Proposition \ref{pRigidity}. A possible reference is the first chapter in Seidel's book \cite{Seidel}. In particular, in an $A_{\infty}$-category we are given a set of objects $X_i$ with a graded vector space $\mathrm{hom} (X_0, X_1)$ for any pair of objects, and composition maps of every order $d\ge 1$
\[
\mathrm{hom} (X_0, X_1) \otimes \mathrm{hom} (X_1, X_2) \otimes \dots \otimes \mathrm{hom}(X_{d-1}, X_d) \to \mathrm{hom} (X_0, X_d)[2-d] 
\]
satisfying the $A_{\infty}$-associativity equations, which we actually need not know precisely here. The important point is that $m_d$ is homogeneous of degree $2-d$. Another important point is cf. \cite{Seidel} that any homotopy unital $A_{\infty}$-category is quasi-isomorphic to a strictly unital one, i.e. we may assume
\[
m_d (a_0\otimes \dots \otimes a_{i-1}\otimes \mathrm{id} \otimes a_{i+1} \otimes \dots \otimes a_d ) = 0, d\ge 3
\]
which means $m_d$, $d\ge 3$, is zero as soon as one of its arguments is a homothetic automorphism of an object. We can now give the 

\begin{proof}(of Proposition \ref{pRigidity}) 
We think of the $\mathcal{L}_{i,t}$ as the objects of our $A_{\infty}$-category. It is clear that
\[
m_2 : \mathrm{hom} (X_0, X_1) \otimes \mathrm{hom} (X_1, X_2) \to \mathrm{hom} (X_0, X_2)
\]
is always the zero map in our case if $X_0$, $X_1$, $X_2$ are pairwise different; in fact, this follows from Lemma \ref{lRigidity}, part (1). Hence it suffices to prove that there is no higher multiplication, i.e. $m_i = 0$ for $i \ge 3$. Then the endomorphism algebra of our category is just a usual graded algebra, and the algebra structure is completely determined and does not deform.

Clearly, $m_d=0$ for $d\ge 5$: in fact, if $i< j<k<l<m <n$, one of the spaces
\begin{gather*}
\mathrm{RHom}^{\bullet} (\mathcal{L}_{i,t},\;  \mathcal{L}_{j,t}), \mathrm{RHom}^{\bullet} (\mathcal{L}_{j,t}, \mathcal{L}_{k,t})\\
\mathrm{RHom}^{\bullet} (\mathcal{L}_{k,t},\;  \mathcal{L}_{l,t}), \mathrm{RHom}^{\bullet} (\mathcal{L}_{l,t}, \mathcal{L}_{m,t}), \; \mathrm{RHom}^{\bullet} (\mathcal{L}_{m,t}, \mathcal{L}_{n,t})
\end{gather*}
is the zero space. Now look at $m_d=4$ (it is helpful to picture the $\mathcal{L}_{i,t}$, $i\neq 3,7$, as objects of one sort, say circles, and picture the $\mathcal{L}_{3,t}$, $\mathcal{L}_{7,t}$ as special, say, boxes).   By Lemma \ref{lRigidity} it follows that the smallest degree of a nonzero element in a space
\[
 \mathrm{hom} (\mathcal{L}_{i,t}, \mathcal{L}_{j,t}) \otimes \mathrm{hom} (\mathcal{L}_{j,t}, \mathcal{L}_{k,t}) \otimes \mathrm{hom} (\mathcal{L}_{k,t}, \mathcal{L}_{l,t}) \otimes \mathrm{hom} (\mathcal{L}_{l,t}, \mathcal{L}_{m, t}) 
\]
for $i<j<k<l<m$ is equal to $5$. But $m_4$ lowers the degree by $2$, and there are no $\mathrm{Ext}^3$'s. The case of $m_3$ follows similarly, but needs some more checking:  we have to look at possible nonzero compositions
\[
 \mathrm{hom} (\mathcal{L}_{i,t}, \mathcal{L}_{j,t}) \otimes \mathrm{hom} (\mathcal{L}_{j,t}, \mathcal{L}_{k,t}) \otimes \mathrm{hom} (\mathcal{L}_{k,t}, \mathcal{L}_{l,t})\to \mathrm{hom} (\mathcal{L}_{i,t}, \mathcal{L}_{l,t})
 \]
 for $i<j<k<l$, and find that the smallest degree of a nonzero element in the left-hand space is always $4$ except in one particular case which we will describe shortly.  However, $m_3$ lowers the degree by $1$, so degree $4$ elements are mapped to zero. Let us consider the particular case, where we compose the potential degree $0$ morphism from $\mathcal{L}_{2,t}$ to $\mathcal{L}_{3,t}$ with a degree $2$ morphism from $\mathcal{L}_{3,t}$ to some $\mathcal{L}_{i,t}$, $3 < i<7$, and then compose with a degree $1$ morphism to go to $\mathcal{L}_{7,t}$: this gives a degree $3$ element to which we may apply $m_3$ to get a degree $2$ element in $\mathrm{Ext}^2(\mathcal{L}_{2,t}, \mathcal{L}_{7,t})$: however, this space is zero by Lemma \ref{lRigidity}, part (5). This completes the proof. 
\end{proof}

\begin{remark}\xlabel{rStrongGenerator}
The arguments of this section have some other consequences which deserve mentioning. 
Consider the complementary category $\mathcal{A}$ to the sequence $(\mathcal{L}_i)$ and put $\mathcal{A}' = \langle \mathcal{A}, \mathcal{O}_X\rangle $. Then $\mathcal{A}'$ contains all elements of $\mathrm{Pic}(X)_{\mathrm{tors}}$, but 
\[
\mathcal{O}_X \oplus \mathcal{O}_{\tau} \oplus \mathcal{O}_{\tau^2}\oplus \mathcal{O}_{\tau^3} \oplus \mathcal{O}_{\tau^4}
\]
is not a strong generator of $\mathcal{A}'$: in fact, the only derived endomorphisms of this object are again homotheties of the respective torsion bundles, and $\mathrm{Ext}^2 (\mathcal{O}_{\tau^i}, \mathcal{O}_{\tau^j})= \mathbb{C}$ for $i\neq j$. Hence the Yoneda algebra $H^* (\mathfrak{D}) $ of this object (with its $A_{\infty}$-structure) has no higher multiplication (composing three degree $2$ morphisms gives something of degree $6$ which cannot be mapped to something nonzero by $m_3$ which has degree $-1$; similarly for the $m_i$, $i\ge 4$). Hence $H^* (\mathfrak{D})$ is a usual graded algebra; but then it easily follows that $\mathrm{HH}_0 ( H^* (\mathfrak{D}) ) = H^* (\mathfrak{D}) / [ H^* (\mathfrak{D}) , H^* (\mathfrak{D}) ]$ has rank $5$; this cannot be if the above were a generator since $\mathrm{HH}_* (\mathcal{A}' )$ has rank $1$. 

This indicates that there must be other objects in $\mathcal{A}$ which are ``not built out of the torsion on $X$", in particular are not of the type $C_{\tau }$. By Theorem 4 of \cite{Orlov09} one may obtain a strong generator of $\mathcal{A}$ by projecting $\mathcal{O}_X \oplus \mathcal{M} \oplus \mathcal{M}^2$ into $\mathcal{A}$, where $\mathcal{M}$ is a very ample line bundle on $X$. We do not know, however, how explicitly this projection is computable, and, moreover, if the resulting complex is still simple enough to allow any interesting conclusions about $\mathcal{A}$.
\end{remark}

\end{document}